\documentclass[12pt]{amsart} \usepackage{amsmath,centernot}
\usepackage{amssymb}
\usepackage[hidelinks]{hyperref}

\newcommand{\Bohr}{\operatorname{Bohr}}

\newcommand{\Cay}{\operatorname{Cay}}
\newtheorem{theorem}{Theorem}
\newtheorem{lemma}[theorem]{Lemma}
\newtheorem{proposition}[theorem]{Proposition}
\newtheorem{corollary}[theorem]{Corollary}

\numberwithin{theorem}{section}
\numberwithin{equation}{section}

\theoremstyle{definition}

\newtheorem{observation}[theorem]{Observation}
\newtheorem{definition}[theorem]{Definition}
\newtheorem{question}[theorem]{Question}
\newtheorem{remark}[theorem]{Remark}

\title[Separating recurrence properties ]{Separating measurable recurrence from strong recurrence via rigidity sequences}

\date{\today}

\author{John T. Griesmer}
\email{jtgriesmer@gmail.com}
\address{Department of Applied Mathematics and Statistics, Colorado School of Mines, Golden, Colorado}
\usepackage{amstext}

\begin{document}

\begin{abstract}
If $G$ is an abelian group, we say $S\subseteq G$ is a \emph{set of recurrence} if for every probability measure preserving $G$-system $(X,\mu,T)$ and every $D\subseteq X$ having $\mu(D)>0$, there is a $g\in S$ such that $\mu(D\cap T^{g}D)>0$.  We say $S$ is a \emph{set of strong recurrence} if for every set $D$ having $\mu(D)>0$ there is a $c>0$ such that $\mu(D\cap T^{g}D)>c$ for infinitely many $g\in S$.    We call $S$ \emph{measure expanding} if for all $g\in G$, the translate $S+g$ is a set of recurrence. A \emph{rigidity sequence}  for $(X,\mu,T)$ is a sequence of elements $s_n\in G$ satisfying $\lim_{n\to\infty} \mu(D\triangle T^{s_n}D)=0$ for all measurable $D\subseteq X$.

For all but countably many countable abelian groups $G$, we prove that if $S\subseteq G$ is measure expanding, there is a sequence of elements $s_n\in S$ such that $\{s_n:n\in \mathbb N\}$ is also measure expanding and for all $t\in G$, $(s_n+t)_{n\in \mathbb N}$ is a rigidity sequence for some free weak mixing measure preserving $G$-system. The special case where $S=G$ proves Conjecture 1.6 of \cite{Ackelsberg}, due to Ackelsberg.

As a consequence, we prove that for every countably infinite abelian group $G$ and every measure expanding set $S\subseteq G$ there is a subset $S'\subseteq S$ such that $S'$ is measure expanding and no translate of $S'$ is a set of strong recurrence.
\end{abstract}

\maketitle

\setcounter{tocdepth}{1}

\tableofcontents

\section{Introduction}

\subsection{Sets of recurrence, uniform distribution, and rigidity sequences}\label{sec:Recurrence} If $G$ is a countable abelian group, a \emph{probability measure preserving $G$-system} (or simply, ``$G$-system'') is a triple $(X,\mu,T)$ where $(X,\mu)$ is a probability measure space and $T$ is an action of $G$ on $X$ by measure preserving transformations, meaning $\mu(T^{g}D)=\mu(D)$ for every measurable $D\subseteq X$ and all $g\in G$.

\begin{definition}
Let $G$ be a countable abelian group and $S\subseteq G$. We say that $S$ is
\begin{enumerate}
  \item[$\bullet$] a \emph{set of recurrence} if for every $G$-system $(X,\mu,T)$ and every $D\subseteq X$ having $\mu(D)>0$, there exists $g\in S$ such that $\mu(D\cap T^{g}D)>0$.

      \smallskip

\item[$\bullet$] \emph{measure expanding} if every translate of $S$ is a set of recurrence, meaning that for all $h\in G$, the set $S+h$ is a set of recurrence.\footnote{The terminology is inspired by the fact that $S$ being measure expanding is equivalent to the condition that for every ergodic $G$-system $(X,\mu,T)$ and all $D\subseteq X$ with $\mu(D)>0$, we have $\mu\bigl(\bigcup_{s\in S}T^sD\bigr)=1$.  However, we do not need this fact for our constructions.}

\smallskip

\item[$\bullet$] a \emph{set of strong recurrence} if for every $G$-system $(X,\mu,T)$ and every $D\subseteq X$ having $\mu(D)>0$, there exists $c>0$ and infinitely many $g\in S$ such that $\mu(D\cap T^{g}D)>c$.

    \smallskip

\item[$\bullet$] a \emph{set of optimal recurrence} if for every $G$-system $(X,\mu,T)$, every $D\subseteq X$ having $\mu(D)>0$, and all $\varepsilon>0$, there is a $g\in S$ such that $\mu(D\cap T^{g}D)>\mu(D)^2-\varepsilon$.
\smallskip

\item[$\bullet$] a \emph{set of rigidity} if $S$ is infinite and there is a nontrivial $G$-system $(X,\mu,T)$ such that  for all measurable $D\subseteq X$ and all $\varepsilon>0$, $\{g\in S: \mu(D \triangle T^{g}D)>\varepsilon\}$ is finite.
 \end{enumerate}
We say  $(s_{n})_{n\in \mathbb N}$ is a \emph{rigidity sequence  for} $(X,\mu,T)$ if $\lim_{n\to \infty} \mu( D\triangle T^{s_{n}}D)=0$ for every measurable $D\subseteq X$.   Note that $S$ is a set of rigidity if and only if every enumeration of $S$ is a rigidity sequence.
\end{definition}

The set $S_2:=\{1, 4, 9, \dots\}$ of perfect squares is a set of recurrence,  as shown independently by Furstenberg \cite{Fdiagonal}  and Sárközy \cite{Sarkozy}.  The set of perfect squares is also a set of strong recurrence, and this fact can be deduced from the proofs in \cite{Fdiagonal} and \cite{Sarkozy}.  However, it is not measure expanding: for every $n\in \mathbb N$, $S_2+n$ is not a set of recurrence (cf.~\cite{FurstenbergPoincare}).

We are motivated to study measure expanding sets partly due to their combinatorial properties, seen in \cite{BjorklundFishPlunnecke} and implicitly in \cite{GriesmerIsr} and \cite{BBF}.

One condition that implies $S$ is measure expanding is the existence of a sequence of elements $s_n\in S$ which is \emph{Hartman uniformly distributed} (Harmtan-u.d.), meaning $\lim_{N\to\infty} \frac{1}{N}\sum_{n=1}^N \chi(s_n)=0$ for every nontrivial character $\chi$ of $G$ (see \S\ref{sec:LCA} for definitions).  For example $(\lfloor n^{5/2}\rfloor)_{n\in \mathbb N}$ is Hartman-u.d., as is every sequence $(\lfloor n^{c}\rfloor)_{n\in \mathbb N}$ where $c>0$, $c\notin \mathbb N$.   There is a plethora of examples in the literature (\cite{Niederreiter},\cite{Wierdl},\cite{NairPoincare1},\cite{BKQW}), hence a plethora of measure expanding sets. Until \cite{GrPop}, every known example of a measure expanding set contained a Hartman-u.d.~sequence.   A sequence $(s_n)_{n\in \mathbb N}$ being Hartman-u.d.~is preserved under translations, and implies $\{s_n:n\in \mathbb N\}$ is a set of optimal recurrence; cf.~\cite[Theorem 1]{NairPoincare1}, \cite{FanSchneiderRecurrence}, \cite{Wierdl}). The previously known examples of measure expanding sets therefore have the property that every one of their translates is a set of optimal recurrence, while the definition of ``measure expanding'' requires only that every translate be a set of recurrence. The main result of \cite{GrPop} found measure expanding sets $S\subseteq \mathbb Z$ with the property that every translate of $S$ is a set of rigidity for some weak mixing $\mathbb Z$-system, and \cite{Ackelsberg} generalized this construction from $\mathbb Z$ to arbitrary countable abelian groups. Consequently these measure expanding sets $S$ are not sets of strong recurrence, by Lemma \ref{lem:Translate} below. In particular such an $S$ cannot contain a Hartman-u.d.~sequence. Theorems \ref{thm:Main0} and Corollary \ref{cor:Main} below extend the results of \cite{GrPop} and \cite{Ackelsberg}, by showing that such examples are ubiquitous: if $S\subseteq G$ is measure expanding, then there is a subset $S'\subseteq S$ which is also measure expanding, and no translate of $S'$ is a set of strong recurrence.  This exhibits some nontrivial structure common to all measure expanding sets.

\subsection{Results}\label{sec:results} The \emph{exponent} of an abelian group $G$ (in additive notation) is the least $n\in \mathbb N$ such that $ng=0$ for every $g\in G$, if such $n$ exists.  Otherwise, we say that $G$ has \emph{infinite exponent}.  Every group of finite exponent is a direct sum of cyclic groups - see Theorem 5.3 in Chapter 3 of \cite{Fuchs} or Theorem 6 in Chapter 8 of \cite{KaplanskyBook} for a proof.

The statements of our results are somewhat complicated for an arbitrary countable abelian group, so we exclude a small class of groups from the initial discussion.  Theorem \ref{thm:MainFinite} covers the excluded groups.

Given a sequence of groups $(G_{n})_{n\in \mathbb N}$, their direct sum is denoted $\bigoplus_{n=1}^{\infty} G_{n}$.

\begin{definition}\label{def:Sigmaq}
  Given $q\in \mathbb N, q\geq 2$,  let $\Sigma_{q}:=\bigoplus_{n=1}^{\infty} \mathbb Z/q\mathbb Z$, meaning $\Sigma_{q}$ is the direct sum of countably many copies of $\mathbb Z/q\mathbb Z$.

  We say a countably infinite abelian group is \emph{reasonably homogeneous} if $G$ either has infinite exponent or $G = \Sigma_{q_1} \oplus \Sigma_{q_2} \oplus \cdots \oplus \Sigma_{q_r}$ for some $r, q_i\in \mathbb N$.
\end{definition}

 A $G$-system $(X,\mu,T)$ is \emph{ergodic} if every  $D\subseteq X$ satisfying $\mu(D\triangle T^{g}D)=0$ for every $g\in G$ has $\mu(D)=0$ or $\mu(D)=1$.  We say $(X,\mu,T)$ is \emph{weak mixing} if the product system $(X\times X,\mu\times \mu,T\times T)$ is ergodic.

We say that $S\subseteq G$ is a \emph{set of strong recurrence for weak mixing systems} if for every weak mixing $G$-system $(X,\mu,T)$ and all $D\subseteq X$ having $\mu(D)>0$, there is a $c>0$ such that $\mu(D\cap T^{g}D)>c$ for infinitely many $g\in S$.

A $G$-system $(X,\mu,T)$ is \emph{free} if for every $h\neq 0\in G$, $\mu(\{x\in X:T^h x = x\})=0$.

Our main results are the following theorem and corollary. Theorem \ref{thm:Main0} proves Conjecture 1.6 of \cite{Ackelsberg}.  The special case of that conjecture where $G$ has finite exponent was resolved by \cite{CherryKQ}.
\begin{theorem}\label{thm:Main0}
  Let $G$ be a reasonably homogeneous countable abelian group.  Then for every measure expanding set $S\subseteq G$, there is a measure expanding set $S'\subseteq S$ such that for all $g\in G$, $S'+g$ is a set of rigidity for some free weak mixing $G$-system.
\end{theorem}
Theorem \ref{thm:Main0} is proved in \S\ref{sec:CompactnessAndProof}, after introducing several building blocks.

The following corollary requires no restriction on $G$.
\begin{corollary}\label{cor:Main}
Let $G$ be a countably infinite abelian group.  If $S\subseteq G$ is measure expanding, then there is a measure expanding set $S'\subseteq S$ such that no translate of $S'$ is a set of strong recurrence for weak mixing systems.  In particular, no translate of $S'$ is a set of strong recurrence.
\end{corollary}
The special cases of Theorem \ref{thm:Main0} and Corollary \ref{cor:Main} where $S=G=\mathbb Z$ are Theorem 2.1 and Corollary 2.3 of \cite{GrPop}.

 Corollary \ref{cor:Main} is proved in \S\ref{sec:Subgroups}.

\begin{remark}\label{rem:DisjointSpectrum} The classification of abelian groups with finite exponent shows that, up to isomorphism, the class of reasonably homogenous groups excludes only countably many countable abelian groups, the exceptional groups having the form $F\oplus \Sigma_{q_1}\oplus \cdots \oplus \Sigma_{q_r}$, where $F$ is a finite abelian group. When $F$ is nontrivial, the statement of Theorem \ref{thm:Main0} no longer holds.  For example, in the group $G=(\mathbb Z/3\mathbb Z)\oplus \Sigma_2$, if a rigidity sequence $(g_n)_{n\in \mathbb N}$ is written as $(g_{n,1},g_{n,2})$ with $g_{n,1}\in \mathbb Z/3\mathbb Z$, then $g_{n,1}=0$ for all but finitely many $n$ -- see Section 7.1 of \cite{Ackelsberg} for a proof and further discussion.   Thus, if  $(g_n)_{n\in \mathbb N}$ is a rigidity sequence for a free weak mixing $G$-system, then $(g_n+(1,0))_{n\in \mathbb N}$ cannot be a rigidity sequence for a free weak mixing $G$-system.  Consequently, there is no subset $S\subseteq G$ having the property that every translate of $S$ is a set of rigidity for some weak mixing $G$-system.
\end{remark}

The phenomenon described above is the only obstruction to every translate of a given sequence being a rigidity sequence, so we recover much of Theorem \ref{thm:Main0} in Theorem \ref{thm:MainFinite}, which covers the previously excluded groups.

\begin{theorem}\label{thm:MainFinite}
  Let $F$ be a finite abelian group and $r\in \mathbb N$, $2\leq q_i\in \mathbb N$.  Let $G=F\oplus (\Sigma_{q_1} \oplus\cdots \oplus \Sigma_{q_r})$, so that $H:= \{0\}\oplus \Sigma_{q_1} \oplus\cdots \oplus \Sigma_{q_r}$ is a finite index subgroup of $G$.

  If $S\subseteq G$ is such that for all $h\in H$, $S+h$ is a set of recurrence, then there is a sequence $(s_n)_{n\in \mathbb N}$ of elements of $S$ such that for all $h\in H$, $(s_n+h)$ is a rigidity sequence for some free weak mixing $G$ system, and $\{s_n+h:n\in \mathbb N\}$ is a set of recurrence.
\end{theorem}

We do not provide an exact characterization of those groups satisfying the conclusion of Theorem \ref{thm:Main0}, as our main interest is in Corollary \ref{cor:Main}.

The proof of Theorem \ref{thm:Main0} combines the ideas behind the construction of Kronecker sets and $K_{q}$-sets (see \cite{Rudin}, Theorem 5.2.2) with the proof of Proposition 5.2 of \cite{GrPop}.   While the present article is self contained, the reader may find our arguments more natural after reading the preceding references and/or \cite{Ackelsberg}.  The main technical fact underlying our proofs is the seemingly innocuous Lemma \ref{lem:Aura}.

\subsection{Rigidity and Fourier transforms of measures}\label{sec:DualEtc}

Let $G$ be a countable abelian group and $\widehat{G}$ its Pontryagin dual (see \S\ref{sec:LCA} for definitions).  The \emph{Fourier transform} $\hat{\sigma}$ of a measure $\sigma$ on $\widehat{G}$ is the function $\hat{\sigma}:G\to \mathbb C$ given by $\hat{\sigma}(g):=\int \chi(g) \,d\sigma(\chi)$.  We say that $\sigma$ is \emph{continuous} if $\sigma(\{\chi\})=0$ for all $\chi\in\widehat{G}$; in other words, $\sigma$ is atomless. Rigidity sequences for weak mixing systems can be characterized entirely in terms of Fourier transforms of measures.  

\begin{lemma}\label{lem:FourierCharacterize}
Let $(s_n)_{n\in \mathbb N}$ be a sequence of elements of $G$.  By \cite[Lemma 2.7 (2)]{Ackelsberg}, the following are equivalent.
  \begin{enumerate}
    \item[(i)] There is a weak mixing $G$-system $(X,\mu,T)$ for which $(s_n)_{n\in \mathbb N}$ is a rigidity sequence.
    \item[(ii)] There is a continuous Borel probability measure $\sigma$ on $\widehat{G}$ with $\lim_{n\to\infty}\hat{\sigma}(s_n)=1$.
  \end{enumerate}
  \end{lemma}
Because we want each $T^g$ with $g\neq 0$ to act nontrivially, we need some more detail relating the support of $\sigma$ to the corresponding $G$-system:
\begin{lemma}\label{lem:FourierCharacterizeFree}
By \cite[Lemma 3.3]{Ackelsberg} the following conditions are equivalent:
\begin{enumerate}
\item[(i)] There is a free weak mixing $G$-system for which $(s_n)_{n\in \mathbb N}$ is a rigidity sequence.
\item[(ii)] There is a continuous Borel probability measure $\sigma$ on $\widehat{G}$ such that $\lim_{n\to\infty}\hat{\sigma}(s_n)=1$ and the subgroup generated by the support of $\sigma$ is dense in $\widehat{G}$.
\end{enumerate}
\end{lemma}

By Lemma \ref{lem:FourierCharacterizeFree}, one may construct a free weak mixing  $G$-system $(X,\mu,T)$ with a rigidity sequence $(s_{n})_{n\in \mathbb N}$ by finding a continuous probability measure $\sigma$ on $\widehat{G}$ such that $\lim_{n\to \infty} \hat{\sigma}(s_{n})=1$ and the support of $\sigma$ generates a dense subgroup of $\widehat{G}$.

\begin{proposition}\label{prop:KroneckerMeasures}
    Let $G$ be reasonably homogeneous countable abelian group.  If $S\subseteq G$ is measure expanding, there is a subset $S'\subseteq S$ and a  continuous probability measure $\sigma$ on $\widehat{G}$ such that
    \begin{enumerate}
      \item[(i)] $S'$ is measure expanding;
     \item[(ii)] $\lim_{n\to \infty} \hat{\sigma}(s_n)=1$, where $(s_n)$ enumerates $S'$;
     \item[(iii)] $\sigma(U)>0$ for every nonempty open $U\subseteq \widehat{G}$.
      \end{enumerate}
\end{proposition}
Proposition \ref{prop:KroneckerMeasures} is proved in \S\ref{sec:SaekiMeasures}.

\subsection{Compactness properties of recurrence, proof of Theorem \ref{thm:Main0}}\label{sec:CompactnessAndProof}
\begin{definition}\label{def:deltaRecurrence}
A set $S\subseteq G$ is a \emph{set of $\delta$-recurrence} if for every $G$-system $(X,\mu,T)$ and every $D\subseteq X$ such that $\mu(D)>\delta$, there exists $g\in S$ such that $\mu(D\cap T^{g}D)>0$.
\end{definition}
Note that $S$ is a set of recurrence if and only if $S$ is a set of $\delta$-recurrence for all $\delta>0$.

\begin{lemma}[cf.~{\cite[Lemma 6.4]{ForrestThesis}}]\label{lem:UniformRecurrence}
   Let $\delta\geq 0$.  If $S\subseteq G$ is a set of $\delta$-recurrence then for all $\delta'>\delta$, there is a finite set $S'\subseteq S$ such that $S'$ is a set of $\delta'$-recurrence.
\end{lemma}

Lemma \ref{lem:UniformRecurrence} is an immediate consequence of Lemma \ref{lem:CayleyCharacterization}. It can also be proved by following the proof of Lemma 6.4 of \cite{ForrestThesis}.

\begin{lemma}\label{lem:ExpandingSelection}
  Let $(S_{n})_{n\in \mathbb N}$ be a sequence of measure expanding subsets of $G$.  Then there is a sequence of finite sets $S_{n}'\subseteq S_{n}$ such that $S':=\bigcup_{n=1}^{\infty} S_{n}'$ is measure expanding.
\end{lemma}

\begin{proof}
Let $(S_n)_{n\in \mathbb N}$ be as in the hypothesis, and enumerate $G$ as $(g_n)_{n\in \mathbb N}$.  For each $n,j\in \mathbb N$ the set $S_{n}+g_{j}$ is a set of recurrence, so apply Lemma \ref{lem:UniformRecurrence} to choose a finite set $S_{n}'\subseteq S_{n}$ so that $S_{n}'+g_{j}$ is a set of $\frac{1}{n}$-recurrence for each $j\leq n$.  Then for each $j\in \mathbb N$ and all  $\delta>0$, the set $S'+g_j:=\bigcup_{n=1}^{\infty} S_{n}'+g_j$ is a set of $\delta$-recurrence.  Thus every translate of $S'$ is a set of recurrence.
\end{proof}

\begin{proof}[Proof of Theorem \ref{thm:Main0}]
  Let $G$ be a reasonably homogeneous countable abelian group, enumerated as $(g_k)_{k\in \mathbb N}$. Let $S\subseteq G$ be measure expanding.  To prove Theorem \ref{thm:Main0}, we construct a sequence of sets $S=S_0\supseteq S_1\supseteq S_2\supseteq S_3\supset \dots$ such that for each $k$, $S_k$ is measure expanding, and $S_k-g_k$ is a set of rigidity for some free weak mixing $G$-system $(X_k,\mu_k,T_k)$.  We then apply Lemma \ref{lem:ExpandingSelection} to find a sequence of finite sets $S_k'\subseteq S_k$ such that $S':=\bigcup_{k\in \mathbb N} S_k'$ is measure expanding.  For each $k$, all but finitely many elements of $S'$ belong to $S_k$, so $S'-g_k$ will be a set of rigidity for $(X_k,\mu_k,T_k)$.  Thus every translate of $S'$ is a set of rigidity for some free weak mixing $G$-system.

  We construct the sequence $S_k$ inductively.  Assume $S_{k-1}$ is defined and is measure expanding.  Then the translate $S_{k-1}-g_k$ is measure expanding as well, so we can apply Proposition \ref{prop:KroneckerMeasures}  to find a measure expanding set $S_k\subseteq S_{k-1}$ and a probability measure $\sigma_k$ on $\widehat{G}$ with full support satisfying $\lim_{n\to\infty} \hat{\sigma}_k(s_n^{(k)}-g_k)=1$, where $(s_n^{(k)})_{n\in \mathbb N}$ enumerates $S_k$.  Since $\sigma_k$ is fully supported,  Lemma \ref{lem:FourierCharacterizeFree}  provides a free weak mixing $G$-system $(X_k,\mu_k,T_k)$ for which $(s_n^{(k)}-g_k)_{n\in \mathbb N}$ is a rigidity sequence.
  \end{proof}

\subsection{Relation to earlier work}

Forrest \cite{Forrest} proved that there is a set of recurrence which is not a set of strong recurrence, in both $\mathbb Z$ and in $\bigoplus_{n=1}^\infty \mathbb Z/2\mathbb Z$; see \cite{McCutcheonAlexandria} for an alternative approach to Forrest's construction.  In \cite[Corollary 2.3]{GrPop} this was improved to show that there is a set $S\subseteq \mathbb Z$ which is measure expanding but no translate of $S$ is a set of strong recurrence.  Corollary \ref{cor:Main} shows that such examples are ubiquitous in every countably infinite abelian group: they can be found within any prescribed measure expanding set.

Fayad and Kanigowski \cite{FaKa} showed that in $\mathbb Z$, there is a rigidity sequence $(s_n)_{n\in \mathbb N}$ for some weak mixing $\mathbb Z$-system which is not a rigidity sequence for any circle rotation.  Theorem 2.1 of \cite{GrPop} improves this to find a sequence $(s_n)_{n\in \mathbb N}$ such that for all $t\in \mathbb Z$, $(s_n+t)_{n\in \mathbb N}$ is a rigidity sequence for some weak mixing system, while $\{s_n+t:n\in \mathbb N\}$ is a set of recurrence for all $t\in \mathbb Z$.  Section 5 of \cite{BadeaGrivauxMatheron} constructs rigidity sequences $(s_n)_{n\in \mathbb N}$ such that $\{s_n: n\in \mathbb N\}$ is dense in the Bohr topology of $\mathbb Z$.  Ackelsberg \cite[Theorem C]{Ackelsberg} generalizes all these prior results on rigidity sequences from $\mathbb Z$ to all countable abelian groups, with the caveat that in some groups, the system for which $(s_n+t)$ is rigid is not  guaranteed to be free.  This leads to \cite[Question 7.6]{Ackelsberg} about Kronecker sets and $K_q$-sets.  While we do not resolve this question, our Theorem \ref{thm:MainFinite} proves Conjecture 1.6 of \cite{Ackelsberg}.  The main result of \cite{CherryKQ} resolves the issue of freeness for countable abelian groups of finite exponent, which answers the ``finite exponent version'' of \cite[Question 7.6]{Ackelsberg}. In the terminology of \cite{Ackelsberg}, our Theorem \ref{thm:MainFinite} says that if $G$ is a countable abelian group, there is a finite index subgroup $H\leq G$ such that if $S\subseteq G$ is measure expanding, then there is a sequence $(r_n)$ of elements of $S$ such that for all $h\in H$, $(r_n-h)$ is a freely rigid-recurrent sequence.  Theorem \ref{thm:Main0} covers those groups where we can take $H=G$.

\subsection{Outline of the article}
Section \ref{sec:LCA} summarizes some standard harmonic analysis.

Section \ref{sec:SaekiMeasures} reviews previous constructions and introduces a special class of measures on compact abelian groups; we call them ``$\mathcal P$-Saeki measures,'' or ``$\mathcal P^{(q)}$-Saeki measures,'' where $\mathcal P$ denotes some distinguished class of subsets of $G$.  These are used to prove Proposition \ref{prop:KroneckerMeasures}  at the end of \S\ref{sec:SaekiMeasures}.   Sections \ref{sec:Uniformity} through \ref{sec:Measures} are dedicated to proving that such measures exist.

Section \ref{sec:Uniformity} introduces Bohr neighborhoods and proves the essential Lemma \ref{lem:Aura}.



Section \ref{sec:UBD} summarizes techniques for transferring recurrence properties between groups. 

Section \ref{sec:BHballs} states a key recurrence property of Hamming balls in powers of cyclic groups.  This is used to prove Lemma \ref{lem:IndependentSets}, which forms the base of our construction in the proof of Proposition \ref{prop:KroneckerMeasures}.

Section \ref{sec:IndependentSets} collects some  facts about finite independent subsets of compact groups.

 Section \ref{sec:Measures} constructs the measures on $\widehat{G}$ introduced in \S\ref{sec:SaekiMeasures}, proving Proposition \ref{prop:KroneckerMeasures}.

Section \ref{sec:arbitrary} contains the proofs of Theorem \ref{thm:MainFinite} and Corollary \ref{cor:Main}.

\subsection{Question}

\begin{question}\label{q:Main} Let $G$ be a countably infinite abelian group.  Is the following statement true?
\begin{center}  \begin{minipage}{0.85\textwidth} For every set of recurrence $S\subseteq G$, there is a set $S'\subseteq S$ such that $S'$ is a set of recurrence but not a set of strong recurrence. \end{minipage}\end{center}
\end{question}
We are unable to resolve this question in any countably infinite abelian group.

\subsection{Acknowledgement.}  An anonymous referee for \emph{Discrete and Continuous Dynamical Systems} contributed several corrections to this article.

\section{Harmonic analysis on LCA groups}\label{sec:LCA}

\subsection{Pontryagin duality}  Let $\Gamma$ be a locally compact abelian group.  A \emph{character} of $\Gamma$ is a continuous homomorphism $\chi:\Gamma\to \mathcal S^{1}$, where $\mathcal S^{1}$ denotes the group of complex numbers of modulus $1$. The characters of $\Gamma$ form a group under pointwise multiplication, denoted $\widehat{\Gamma}$.  The group $\widehat{\Gamma}$ is called the \emph{dual of} $\Gamma$ (sometimes ``Pontryagin dual''), and is endowed with the topology of uniform convergence on compact subsets.  This means that when $\Gamma$ is discrete, $\widehat{\Gamma}$ has the topology of pointwise convergence, and when $\Gamma$ is compact, $\widehat{\Gamma}$ has the topology of uniform convergence.

\begin{proposition}\label{prop:Pduality}
  Let $\Gamma$ be a locally compact abelian group.

\begin{enumerate}
  \item   $\widehat{\Gamma}$ is a locally compact abelian group -- \cite[Theorem 4.2]{Folland}.

 \item $\widehat{\widehat{\Gamma}}$ is topologically isomorphic to $\Gamma$: the map $\phi: \Gamma\to \widehat{\widehat{\Gamma}}$ given by $\phi(g)=e_{g}$, where $e_{g}(\chi)=\chi(g)$, is both a homeomorphism of topological spaces and an isomorphism of groups -- \cite[Theorem 4.31]{Folland}.

 \item $\widehat{\Gamma}$ is discrete if and only if $\Gamma$ is compact, and vice versa -- \cite[Proposition 4.4]{Folland}.

 \end{enumerate}

\end{proposition}

In general, if $G$ is a countably infinite discrete abelian group, then $\widehat{G}$ is an uncountable compact metrizable abelian group. See \cite{Rudin}, \cite{MorrisLCA}, or \cite{Folland} for further exposition on Pontryagin duality.

\subsection{Notation}  We will write $e(t)$ for $e^{2\pi i t}$.  For $k\in \mathbb N$, let $Z_{k}$ denote the group of $k^{\text{th}}$ roots of unity, so $Z_{k}$ is the cyclic subgroup of $\mathcal S^{1}$ generated by $e(1/k)$.

\subsection{Direct sums and products of cyclic groups}\label{sec:DirectSumsProducts}

If $G$ is a finite cyclic group $\mathbb Z/k \mathbb Z$, then the dual of $G$ is isomorphic to $G$: $\widehat{G}$ is generated by the homomorphism $e_{1}$, where $e_{1}(n+k\mathbb Z)=e(n/k)$.

The following lemma is an immediate consequence of Theorem 2.2.3 of \cite{Rudin}.

\begin{lemma}\label{lem:Products}
   Let $\{G_{j}:j\in I\}$ be a collection of discrete groups and let $G:=\bigoplus_{j\in I} G_{j}$ be their direct sum, also endowed with the discrete topology. Then $\widehat{G}$ is isomorphic to $\prod_{j\in I} \widehat{G}_{j}$, with the product topology.  Furthermore, if $g=(g_{j})_{j\in I}\in G$ and $\chi=(\chi_{j})_{j\in I}\in \widehat{G}$, evaluation is given by $\chi(g)=\prod_{j\in I} \chi_{j}(g_{j})$.
\end{lemma}
Note: the evaluation is well defined as $g_{j}=0$ for all but finitely many $j\in I$.

For $q\in \mathbb N$, let $\Sigma_{q}:=\bigoplus_{n=1}^{\infty} \mathbb Z/q\mathbb Z$, as in Definition \ref{def:Sigmaq}. Let $D_{q}:=\prod_{n=1}^{\infty} \mathbb Z/q\mathbb Z$, with the product topology. Then Lemma \ref{lem:Products} says that $\widehat{\Sigma}_{q}=D_{q}$.  Writing an element of $\Sigma_{q}$ as $a=(a_{n})_{n\in \mathbb N}$ and an element of $D_{q}$ as $\chi=(d_{n})_{n\in \mathbb N}$, evaluation is given by
$\chi(a) = e\bigl(\sum_{n\in \mathbb N}a_{n}d_{n}/q\bigr)$.

\subsection{Kronecker sets and \texorpdfstring{$K_q$}{Kq}-sets}
If $G$ is an LCA group with dual $\widehat{G}$, we say that $K\subseteq \widehat{G}$ is a \emph{Kronecker set} if for every continuous $f:K\to \mathcal S^1$ and all $\varepsilon>0$, there is a $g \in G$ such that $|f(\chi)-\chi(g)|<\varepsilon$ for all $\chi\in K$.  For a fixed $q\in \mathbb N$, we say that $K$ is a \emph{$K_q$-set} if this only holds for all continuous functions $f:K\to Z_q$.  See Chapter 5 of \cite{Rudin} for exposition.

\begin{remark}
	This definition differs slightly from the one in \cite{Rudin}, where a Kronecker set is defined to be a subset of $G$.  The two definitions are equivalent by Pontryagin duality.  
\end{remark}

\section{Kronecker-like sets and measures}\label{sec:SaekiMeasures}

After outlining some previous constructions in the direction of Theorem \ref{thm:Main0}, we introduce a special class of measures in Definition \ref{def:SaekiMeasure}.  Lemma \ref{lem:R4dot} asserts that these measures exist, and will be proved in \S\ref{sec:Measures}.  We prove Proposition \ref{prop:KroneckerMeasures} at the end of this section.

\subsection{Prior constructions} \cite{GrPop} and \cite{Ackelsberg} use the following strategy to construct measure expanding sets which are also sets of rigidity.  For a countable abelian group $G$, let $K\subseteq \widehat{G}$ be a Kronecker set (if $G$ has infinite exponent) or a $K_q$-set (if $G$ has finite exponent) and let $\sigma$ be a continuous Borel probability measure supported on $K$.  Then for every continuous $f:\widehat{G}\to \mathcal S^1$ (or $f:\widehat{G}\to Z_q$, if $G$ has finite exponent) and all $\varepsilon>0$, the set
\[
S_{f,\varepsilon}:=\{g\in G: \textstyle\int |f(\chi)-\chi(g)|\, d\sigma(\chi)<\varepsilon\}
\]
is measure expanding (\cite[Proposition 5.2]{GrPop} for $G=\mathbb Z$, \cite[Corollary 6.3]{Ackelsberg} in general).  Together with compactness properties of sets of recurrence, this can be used to construct $(s_n)_{n\in \mathbb N}$ such that $\{s_n:n\in \mathbb N\}$ is measure expanding while $\lim_{n\to\infty} \int |1- \chi(s_{n})|\, d\sigma(\chi)=0$.  The latter condition is equivalent to $\lim_{n\to\infty} \hat{\sigma}(s_n)=1$, since $\sigma$ is a probability measure.

One would like to repeat this argument to prove Proposition \ref{prop:KroneckerMeasures}, by selecting the $s_n$ from the ambient measure expanding set $S$ in the hypothesis.  This would require that $S_{f,\varepsilon}\cap S$ is measure expanding, and that is often not the case: when $\sigma$ is continuous, the complement $G\setminus S_{f,\varepsilon}$ will have Banach density $1$, and then it is easy to verify that $S:=G\setminus S_{f,\varepsilon}$ itself is measure expanding, while $S_{f,\varepsilon}\cap S$ is empty.  But all we really need is that, given a measure expanding set $S$, there exists at least one continuous measure $\sigma$ with the property that for every $f:\widehat{G}\to \mathcal S^1$ and all $\varepsilon>0$, $S_{f,\varepsilon}\cap S$ is measure expanding.  Lemma \ref{lem:R4dot} says that such measures exist.

\begin{definition}\label{def:SaekiMeasure}
  If $G$ is a countable abelian group and $\mathcal P$ is a collection of subsets of $G$, we say that a positive Borel measure $\sigma$ on $\widehat{G}$ is a $\mathcal P$-\emph{Saeki measure} if for all $\varepsilon>0$ and every continuous $f:\widehat{G} \to \mathcal S^{1}$,
  \begin{equation}\label{eqn:SaekiMeasureDef}
\{g\in G: \textstyle\int|f(\chi)-\chi(g)| \,d\sigma(\chi)<\varepsilon \}\in \mathcal P.
  \end{equation}
   If $q\in \mathbb N$, we say that $\sigma$ is a $\mathcal P^{(q)}$-Saeki measure if for all $\varepsilon>0$ and every measurable $f:\widehat G \to Z_{q}$, the inclusion (\ref{eqn:SaekiMeasureDef}) is satisfied.
\end{definition}

\begin{remark}
  In the above terminology, Saeki \cite{Saeki} proved that if $G$ is a countably infinite abelian group, $\mathcal B$ is the collection of subsets of $G$ dense in the Bohr topology, and $K\subseteq \widehat{G}$ is a Kronecker set (resp.~$K_q$-set),  then every continuous measure supported on $K$ is a $\mathcal B$-Saeki measure (resp.~$\mathcal B^{(q)}$-Saeki measure).  For $G=\mathbb Z$, \cite[Proposition 5.2]{GrPop} improves the this result by replacing  $\mathcal B$ with the class of measure expanding subsets of $\mathbb Z$. This improvement was generalized to all countable abelian groups by \cite[Corollary 6.3]{Ackelsberg}.
\end{remark}

\begin{remark}
  The only function $f$ we use in the proof of Proposition \ref{prop:KroneckerMeasures} is the constant $f\equiv 1$, so we could make a weaker definition of ``$\mathcal P$-Dirichlet'' measure, insisting only that (\ref{eqn:SaekiMeasureDef}) holds for $f\equiv 1$.  This doesn't seem to save any effort in the proof, so we use the stronger definition.
\end{remark}

Given a measure expanding set $S\subseteq G$, we let  $\mathcal R^{\bullet}S$ denote the collection of subsets of $S$ which are measure expanding.

\begin{lemma}\label{lem:R4dot}  Let $G$ be a reasonably homogenous countable abelian group.  Suppose  $S\subseteq G$ is measure expanding and $W\subseteq \widehat{G}$ is open and nonempty.
   \begin{enumerate}
      \item[1.] If $G$ has infinite exponent, then there is a continuous $\mathcal R^{\bullet}S$-Saeki probability measure supported on $W$.
      \item[2.]  If $r, q_{i} \in \mathbb N$, $q_{i}\geq 2$, $G=\Sigma_{q_1}\oplus \cdots \oplus \Sigma_{q_r}$ and $q$ is the exponent of $G$, there is a continuous $\mathcal R^{\bullet}S^{(q)}$-Saeki probability measure supported on $W$.
   \end{enumerate}
\end{lemma}

Lemma \ref{lem:R4dot} is proved in \S\ref{sec:Construction}.

\begin{corollary}\label{cor:OneMeasure}
Let $G$ be a reasonably homogeneous countable abelian group. Suppose  $S\subseteq G$ is measure expanding and $W\subseteq \widehat{G}$ is open and nonempty. Then there is a continuous Borel probability measure $\sigma$ supported on $W$ and a sequence $(s_n)$ of elements of $S$ such that $\lim_{n\to\infty} \hat{\sigma}(s_n)=1$, and $S':=\{s_n:n\in \mathbb N\}$ is measure expanding.
\end{corollary}

\begin{proof}
Let $W\subseteq \widehat{G}$ be open.  By Lemma \ref{lem:R4dot}, we let $\sigma$ be a continuous $\mathcal R^{\bullet}S$-Saeki probability measure, or a continuous $\mathcal R^{\bullet}S^{(q)}$-Saeki probability measure, supported on $W$.   Let $S_n:= \{s\in S: \int |1-\chi(s)|\, d\sigma(\chi)<\tfrac{1}{n}\}$. Each $S_n$ is measure expanding, as seen by setting $f \equiv 1$ in (\ref{eqn:SaekiMeasureDef}).  Lemma \ref{lem:ExpandingSelection} provides finite sets $S_n'\subseteq S_n$ such that $S':=\bigcup_{n\in \mathbb N} S_n'$ is measure expanding.  Enumerating $S'$ as $(s_n)_{n\in \mathbb N}$ in any order, we see that $\lim_{n\to\infty} \int |1-\chi(s_n)|\, d\sigma(\chi)=0$.  Since $\sigma(\widehat{G})=1$, this is equivalent to $\lim_{n\to\infty}\hat{\sigma}(s_n)=1$.
\end{proof}

\begin{proof}[Proof of Proposition \ref{prop:KroneckerMeasures}]
Assuming $G$ is a reasonably homogeneous countable abelian group, the countability of $G$ implies $\widehat{G}$ is metrizable.  Being compact as well, $\widehat{G}$ has a countable base $\{W_{n}:n\in \mathbb N\}$ for its topology.

We will construct a sequence of sets $S\supseteq S_1\supset S_2 \supset S_3\supset \dots $ and continuous measures $\sigma_k$ supported on $W_k$ such that each $S_k$ is measure expanding, while enumerating $S_k$ as $(s^{(k)}_n)_{n\in\mathbb N}$ we have \begin{equation}\label{eqn:SigmakRigid}\lim_{n\to\infty} \hat{\sigma}_k(s^{(k)}_{n})=1.
\end{equation}  With these $\sigma_{k}$ in hand, we let $\sigma = \sum_{k\in \mathbb N} 2^{-k}\sigma_{k}$.  Then $\sigma$ will assign positive measure to each nonempty open set, since $\sigma_k(W_{k})>0$ for every $k$, and every nonempty open set contains at least one of the $W_k$.  Continuity of $\sigma$ is immediate from continuity of each $\sigma_{k}$. Thus $\sigma$ satisfies part (iii) of Proposition \ref{prop:KroneckerMeasures}.

To construct $S_1$ and $\sigma_1$, apply Corollary \ref{cor:OneMeasure} to $S$.  For $k>1$, assume $S_{k-1}$ and $\sigma_{k-1}$ are defined, and apply Corollary \ref{cor:OneMeasure} with $S_{k-1}$ in place of $S$.  Having defined $S_k$ and $\sigma_k$, we define $S'$ by applying Lemma \ref{lem:ExpandingSelection} to find a sequence of finite sets $S_k'\subseteq S_k$ such that $S':=\bigcup_{k\in \mathbb N} S_k'$ is measure expanding. Part (i) of Proposition \ref{prop:KroneckerMeasures} is verified.

To verify part (ii) of Proposition \ref{prop:KroneckerMeasures}, enumerate $S'$ as $(s_n)_{n\in \mathbb N}$.  Then (\ref{eqn:SigmakRigid}) implies $\lim_{n\to\infty} \hat{\sigma}_k(s_n)$ for every $k$, as all but finitely many elements of $S'$ belong to $S_k$.   From this and the definition of $\sigma$ we conclude that $\lim_{n\to\infty} \hat{\sigma}(s_n)=1$.
\end{proof}

\section{Bohr topology and properties of \texorpdfstring{$R_{0}(T,D)$}{R0TD}}\label{sec:Uniformity}
For this section, fix a countable abelian group $G$.  Given a $G$-system $(X,\mu,T)$, a set $D\subseteq X$ with $\mu(D)>0$ and a constant $c\geq 0$, we consider the set
\[R_{c}(T;D):=\{g\in G: \mu(D\cap T^{g}D)>c\}.\]
Note that $S\subseteq G$ is a set of recurrence if and only if for every $G$ system $(X,\mu,T)$ and every $D$ with $\mu(D)>0$, $S\cap R_{0}(T;D)\neq \varnothing$.

\subsection{Bohr topology and group rotations}\label{sec:Bohr} Let $\mathbb T$ denote $\mathbb R/\mathbb Z$.  Representing  $t\in \mathbb T$ by $\tilde{t}\in [0,1)$, let $\|t\|:=\min\{|\tilde t-n|:n\in \mathbb Z\}$.   

Let $G$ be a countable discrete abelian group.  If $d\in \mathbb N$, $\eta>0$, and $C=\{\rho_1,\dots,\rho_{d}\}$ is a set of homomorphsims from $G$ to $\mathbb T$, the corresponding \emph{Bohr-$(d,\eta)$ set} is
\begin{equation}\label{eqn:BohrDef}
\Bohr(C,\eta):=\{g\in G:\max_{j\leq d} \|\rho_{j}(g)\|<\eta\}.
  \end{equation}For $h\in G$, the corresponding \emph{Bohr neighborhood of $h$} is $h+\Bohr(C,\eta)$. 

  Since $\mathbb T$ is isomorphic to $\mathcal S^1$ as a topological group, we can also define Bohr neighborhoods in terms of characters: given $\chi_1,\dots,\chi_d\in \widehat{G}$ and $\varepsilon>0$,
  \[\{g\in G: \max_{j\leq d} |\chi_j(g)-1|<\varepsilon\}\] is a Bohr neighborhood of $0$ in $G$.

A \emph{group rotation $G$-system} $(Z,m,T_{\rho})$ is a $G$-system where $Z$ is a compact abelian group, $m$ is Haar probability measure on $Z$, and $\rho:G\to Z$ is a homomorphism.  The action $T_{\rho}$ is defined by $T_{\rho}^{g}z = z+\rho(g)$.

\begin{observation}\label{obs:BohrIsReturn}
   Given a set $C$ of homomorphisms $\rho_1,\dots,\rho_d:G\to \mathbb T$ and  $\eta>0$, $\Bohr(C,\eta)$ contains a set of the form $R_{0}(T;D)$ for a $G$-system $(X,\mu,T)$ and some $D\subseteq X$ with $\mu(D)>0$.  In fact, $\Bohr(C,\eta)=R_{0}(T;D)$, where  $(X,\mu,T)$ is the group rotation $G$-system $(Z,m,T_{\rho})$, $Z=\mathbb T^{d}$, $\rho=(\rho_{1},\dots,\rho_{d})$, $m$ is Haar probability measure on $\mathbb T^{d}$, and $D=\{(x_1,\dots,x_d)\in \mathbb T^{d}: \max_{j\leq d} \|x_j\|<\eta/2\}$.  
\end{observation}

\subsection{Properties of \texorpdfstring{$R_{0}(T;D)$}{R0TD}}\label{sec:R0}

A key identity for $G$-systems is
\begin{equation}\label{eqn:MPIdentity}
  \mu(D\cap T^{g}E)=\mu(T^{h}D\cap T^{g+h}E) \qquad \text{for all $g,h\in G$ and all $D, E\subseteq X$.}
\end{equation}
  Equation (\ref{eqn:MPIdentity}) follows from the identity $T^{-h}(D\cap E)=T^{-h}D\cap T^{-h}E$ and the fact that $T$ preserves $\mu$.

The following lemma is essential in the proof of Lemma \ref{lem:IndependentSets},  which forms the core of our main argument.

\begin{lemma}\label{lem:Aura}
Let $\delta>0$.  If  $S\subseteq G$ measure expanding, $E\subseteq G$ is a set of $\delta$-recurrence, and $(X,\mu,T)$ is a $G$-system with $D\subseteq X$ having $\mu(D)>0$, then $S\cap (E+R_{0}(T;D))$ is a set of $\delta$-recurrence.  In particular, when $B$ is a Bohr neighborhood of $0$, $S\cap (E+B)$ is a set of $\delta$-recurrence.
\end{lemma}

\begin{proof}
  Let $E$ and $S$ be as in the hypothesis.  Let $(Y,\nu,Q)$ be a $G$-system and let $C\subseteq Y$ have $\nu(C)>\delta$.   To prove the lemma it suffices to find $e\in E$ and $g\in R_{0}(T;D)$ such that $e+g\in S$ and $\nu(C\cap Q^{e+g}C)>0$.

  Since $E$ is a set of $\delta$-recurrence, choose $e\in E$ such that $C_{e}:=C\cap Q^{e}C$ has $\nu(C_{e})>0$.  Consider the product system $(X\times Y,\mu\times \nu, T\times Q)$.  Since $S$ is measure expanding, the translate $S-e$ is a set of recurrence, so there is a $g\in S-e$ such that
  \[
  \mu\times \nu\bigl((D\times C_{e})\cap (T\times Q)^{g}(D\times C_{e})\bigr)>0.
  \]   The above inequality implies  $g\in R_{0}(T,D)$ and that
  \[
  \nu(C\cap Q^{e}C\cap Q^{g}C\cap Q^{e+g}C)>0,
  \] which in turn implies that $\nu(C\cap Q^{e+g}C)>0$, as desired.  Since $g\in S-e$ we have $e+g\in S$.

  The last assertion of the lemma now follows from Observation \ref{obs:BohrIsReturn}.
\end{proof}

\section{Upper Banach density and Cayley graphs}\label{sec:UBD}

The language of Cayley graphs is useful for transferring recurrence properties between groups.  The main result here is Lemma \ref{lem:CopyCayley}, which will be used to prove Lemma \ref{lem:Prelim}.

\subsection{Upper Banach density} Let $G$ be a countable abelian group.  A sequence $(\Phi_{n})_{n\in \mathbb N}$ of finite subsets of $G$ is a \emph{F{\o}lner sequence} if $\lim_{n\to \infty}\frac{|\Phi_{n}\triangle (\Phi_{n}+g)|}{|\Phi_{n}|}=0$ for every $g\in G$.

Every countable abelian group has a F{\o}lner sequence: by \cite{FolnerBanach}, the existence of a F{\o}lner sequence for a countable group $G$ is equivalent to the amenability of $G$, and every abelian group is amenable.

If $\mathbf \Phi=(\Phi_{n})_{n\in \mathbb N}$ is a F{\o}lner sequence for $G$ and $A\subseteq G$, the \emph{upper density of $A$ with respect to $\mathbf{\Phi}$} is $\bar{d}_{\mathbf{\Phi}}(A):=\limsup_{n\to \infty} \frac{|A\cap \Phi_{n}|}{|\Phi_{n}|}$.  The \emph{upper Banach density} of $A$ is \[d^*(A):=\sup\{\bar{d}_{\mathbf{\Phi}}(A): \mathbf{\Phi} \text{ is a F{\o}lner sequence}\}.\]
We need the following characterization of upper Banach density; cf.~\cite[Lemma 3.3]{BBF}.

\begin{lemma}\label{lem:UBD}
Let $A\subseteq G$.  For all finite $F\subseteq G$ there is a $g\in G$ such that $|A\cap (F+g)|\geq d^*(A)|F|$.
  \end{lemma}
  Actually \cite[Lemma 3.3]{BBF} proves only that if $\beta<d^*(A)$, then there is a $g$ such that $|A\cap (F+g)|\geq \beta|F|$.  But since $|A\cap (F+g)|$ is integer-valued, the supremum is attained.

The following version of Furstenberg's correspondence principle will help demonstrate that being a set of recurrence is approximated by finite subsets.

\begin{lemma}\label{lem:Correspondence} Let $(\Phi_n)_{n\in\mathbb N}$ be a F{\o}lner sequence and let $(A_n)_{n\in \mathbb N}$ be a sequence of subsets of $G$ such that $\lim_{n\to\infty} \frac{|A_n\cap \Phi_n|}{|\Phi_n|}=\delta$.  Then there is a $G$-system $(X,\mu,T)$ and $D\subseteq X$ with $\mu(D)\geq \delta$ such that for all $h\in G$,
\begin{equation}\label{eqn:Correspondence}
\limsup_{n\to\infty} \frac{|A_n\cap (A_n+h)\cap \Phi_n|}{|\Phi_n|} \geq \mu(D\cap T^{h}D).
\end{equation}
\end{lemma}

\begin{proof}
Assume $\Phi_{n}$ and $A_{n}$ are as in the hypothesis.  Let $X = \{0,1\}^G$ with the product topology, so that $X$ is compact.  Write elements of $X$ as functions $x:G\to \{0,1\}$, and let $T:X\to X$ be the shift action: $(T^gx)(h)=x(h+g)$.  Let $D=\{x\in X:x(0)=1\}$, so that $D$ is clopen. Let $y_n=1_{A_n}\in X$. Define a sequence of Borel probability measures $\mu_n$ on $X$ by $\mu_n(E) = \frac{1}{|\Phi_n|}|\{g\in \Phi_n: T^gy_n\in E\}|$, so that $\mu_n(X)=1$ for every $n$, and 
\begin{equation}\label{eqn:mun}
	\mu_n(D)=|\{g\in \Phi_n: 1_{A_n}(g)=1\}|/|\Phi_n|=|A_n\cap \Phi_n|/|\Phi_n|.
\end{equation}  Let $\mu$ be any weak$^*$ limit of $(\mu_n)_{n\in \mathbb N}$, so that $\mu$ is again a probability measure. By (\ref{eqn:mun}) have $\mu(D)=\lim_{n\to\infty}\mu_{n}(D)\geq \delta$. It is easy to check that $\mu$ is $T$-invariant, as $|\int f\, d\mu_n - \int f\circ T^h\, d\mu_n| \leq 2\sup |f| |\Phi_n \triangle (\Phi_n+h)|/|\Phi_n|$ for every $n$.

To prove (\ref{eqn:Correspondence}), note that $T^gy_n \in D\cap T^{h}D$ if and only if $y_n(g)=1$ and $y_n(g-h)=1$, and this is equivalent to $g\in A_n$ and $g-h\in A_n$, which means $g\in A_n \cap (A_n+h)$. Thus \[\mu(D\cap T^{h}D)=\lim_{k\to\infty} \frac{|A_{n_k}\cap (A_{n_k}+h)\cap \Phi_{n_k}|}{|\Phi_{n_k}|}, \]
where $n_k$ is any sequence with $\mu_{n_k}(D\cap T^hD)$ converging to $\mu(D\cap T^hD)$.
\end{proof}

Recall the term ``set of $\delta$-recurrence'' from Definition \ref{def:deltaRecurrence}.

\begin{corollary}\label{cor:FolnerRecurrence}
  If $S$ is a set of $\delta$-recurrence, then for all $\delta'>\delta$ and every F{\o}lner sequence $(\Phi_n)_{n\in \mathbb N}$, there is an $n$ such that $(A-A)\cap S\neq \varnothing$ whenever $|A\cap \Phi_n|\geq \delta'|\Phi_n|$.
\end{corollary}
\begin{proof}
  Assume, to get a contradiction, that $S$ is a set of $\delta$-recurrence, $\delta'>\delta$, $(\Phi_n)$ is a F{\o}lner sequence, and for all $n\in \mathbb N$, there is an $A_n\subseteq \Phi_n$ with $|A_n\cap \Phi_n|\geq \delta'|\Phi_n|$ such that $(A_n-A_n)\cap S=\varnothing$.  Then $A_n\cap (A_n+h)=\varnothing$ for all $h\in S$, and $\lim_{n\to\infty} \frac{|A_n\cap (A_n+h)\cap \Phi_n|}{|\Phi_n|}=0$.  Lemma \ref{lem:Correspondence} then provides $G$-system $(X,\mu,T)$ and $D\subseteq X$ with $\mu(D)\geq \delta'$ such that $\mu(D\cap T^{h}D)=0$ for all $h\in S$.  This contradicts the assumption that $S$ is a set of $\delta$-recurrence.
\end{proof}

\begin{lemma}\label{lem:DeltaDensityRecurrence}
  Let $G$ be a countable abelian group and $S\subseteq G$.  Then $S\subseteq G$ is a set of $\delta$-recurrence if and only if $S\cap (A-A)\neq \varnothing$ for all $A\subseteq G$ having $d^{*}(A)>\delta$.
\end{lemma}

\begin{proof}  The implication ``If $S$ is a set of $\delta$-recurrence then $S\cap (A-A)\neq \varnothing$ for all $A\subseteq G$ having $d^*(A)>\delta$'' follows from Corollary \ref{cor:FolnerRecurrence}. To prove the reverse implication, let $S$ be such that $S\cap (A-A)\neq \varnothing$ for all $A\subseteq G$ having $d^*(A)>  \delta$.  To prove that $S$ is a set of $\delta$-recurrence, let $(X,\mu,T)$ be a measure preserving $G$-system, let $D\subseteq X$ have $\mu(D)>\delta$, and let $(\Phi_n)_{n\in \mathbb N}$ be a F{\o}lner sequence.

Let $Q_{0}:=\{g\in G: \mu(D\cap T^{g}D)=0\}$. Let $C:=\bigcup_{g\in Q_{0}} D\cap T^{g}D$ and $D':=D\setminus C$.  Note that $D'\subseteq D$, $\mu(D')=\mu(D)$, and if $D'\cap T^gD'\neq \varnothing$, then $\mu(D\cap T^gD)>0$. So there is no loss of generality in replacing $D$ with $D'$.

Let $f=1_D$.  By Fatou's lemma,
\[
\limsup_{n\to\infty} \int\frac{1}{|\Phi_n|}\sum_{g\in \Phi_n} 1_D\circ T^g\, d\mu \leq \int \limsup_{n\to\infty} \frac{1}{|\Phi_n|}\sum_{g\in \Phi_n} 1_D\circ T^g d\mu.
\]
Since $T$ preserves $\mu$, the left-hand side above equals $\mu(D)$. Then the integral on the right-hand side is at least $\mu(D)$,  $\limsup_{n\to\infty} \frac{1}{|\Phi_n|}\sum_{g\in \Phi_n} 1_D(T^gx)\geq \mu(D)$ for some $x\in X$.  Fixing this $x$ and letting $A:=\{g\in G: T^gx\in D\}$, we have $\frac{1}{|\Phi_n|}\sum_{g\in \Phi_n} 1_D(T^gx)=\frac{|A\cap \Phi_n|}{|\Phi_n|}$.  Thus $\bar{d}_{\Phi}(A)\geq \mu(D)$, so $d^*(A)\geq \mu(D)$. Our assumption on $S$ implies $(A-A)\cap S\neq \varnothing$, so we can write $a_2-a_1=s$ for some $a_i\in A$, $s\in S$. From this we will deduce that $D\cap T^{s}D\neq \varnothing$. To do so, note that we have $a_2\in A$ and $a_2-s\in A$. This means $T^{a_2-s}x\in D$ and $T^{a_2}x\in D$, so $T^{a_2}x\in D\cap T^{s}D$, meaning $D\cap T^{s}D\neq \varnothing$.  By our replacement of $D$ with $D'$ above, we have $\mu(D\cap T^{s}D)>0$.
\end{proof}

\subsection{Cayley graphs}\label{sec:CayleyGraphs}

If $\mathcal G = (V,\mathcal E)$ is a graph with vertex set $V$ and edge set $\mathcal E$, we say $A\subseteq V$ is \emph{independent in $\mathcal G$} if no two elements of $A$ are joined by an edge of $\mathcal G$. When $V$ is finite, the \emph{independence ratio} of $\mathcal G$ is $\frac{1}{|V|}\max\{|A|:A\subseteq V \text{ is independent in } \mathcal G\}$.  If $V'\subseteq V$, the \emph{subgraph of $\mathcal G$ induced by} $V'$ has vertex set $V'$ and edge set $\{\{v,w\}: \{v,w\}\in E, v,w\in V'\}$.

Given a group $G$ and $S\subseteq G$, we let $\Cay(S)$ denote the \emph{Cayley graph} determined by $S$, which has vertex set $G$ and edges $\{g, g'\}$, where $g'-g\in \pm S$.

\begin{lemma}\label{lem:CayleyCharacterization}
Let $G$ be a countable abelian group and $S\subseteq G$.
\begin{enumerate}
    \item[(i)] If $\Cay(S)$ has a finite subgraph with independence ratio $\leq \delta$, then $S$ is a set of $\delta$-recurrence.	

    \item[(ii)] If $S$ is a set of $\delta$-recurrence, then for all $\delta'>\delta$, $\Cay(S)$ has a finite subgraph with independence ratio $\leq \delta'$.
\end{enumerate}
\end{lemma}
Lemma \ref{lem:UniformRecurrence} follows from Lemma \ref{lem:CayleyCharacterization}: a finite subgraph of $\Cay(S)$ is a subgraph of $\Cay(S')$ for some finite $S'\subseteq S$.  Then if $\Cay(S')$ has independence ratio $\leq \delta'$, $S'$ is a set of $\delta'$-recurrence.
\begin{proof}
  (i) Suppose $\Cay(S)$ has a finite subgraph with vertex set $V=\{g_1,\dots, g_k\}$ having independence ratio $\leq \delta$.  We will use Lemma \ref{lem:DeltaDensityRecurrence} to prove $S$ is a set of $\delta$-recurrence.
  If $A\subseteq G$ has $d^*(A)>\delta$, then by Lemma \ref{lem:UBD} there is a translate $V+t$ of $V$ such that $|A\cap (V+t)|> \delta |V|$.  Thus $|(A-t)\cap V|> \delta|V|$, so $A-t$ cannot be independent in $\Cay(S)$, and there are two elements $a-t,a'-t\in A-t$ such that $(a-t)-(a'-t)\in S$.  This means $a,a'\in A$ and $a-a'\in S$.  Since $A$ was an arbitrary subset of $G$ having $d^*(A)>\delta$, Lemma \ref{lem:DeltaDensityRecurrence} implies $S$ is a set of $\delta$-recurrence.

  (ii) Now suppose $S\subseteq G$ is a set of $\delta$-recurrence and $\delta'>\delta$. Let $(\Phi_n)_{n\in \mathbb N}$ be a F{\o}lner sequence.  By Corollary \ref{cor:FolnerRecurrence}, there is an $n$ such that for every $A\subseteq \Phi_n$ having $|A\cap \Phi_n|\geq \delta' |\Phi_n|$, we have $(A-A)\cap S\neq \varnothing$. This means the finite subgraph of $\Cay(S)$ induced by $\Phi_n$ has independence ratio $\leq \delta'$.
\end{proof}

\begin{lemma}\label{lem:CopyCayley} Let $G$ be an abelian group, $K$ a topological abelian group, and let $\rho:G\to K$ be a homomorphism.   If $U\subseteq K$ is open, then every finite subgraph of $\Cay(U)$ with vertices in $\overline{\rho(G)}$ has an isomorphic copy in $\Cay(\rho^{-1}(U))$.

Consequently, if $\Cay(U)$ has a finite subgraph with vertices in $\overline{\rho(G)}$ and independence ratio $\delta$, then $\Cay(\rho^{-1}(U))$ has a finite subgraph with independence ratio $\delta$, and $\rho^{-1}(U)$ is a set of $\delta$-recurrence.
\end{lemma}

\begin{proof}
  To prove the first statement of the lemma it suffices to prove that if $V$ is a finite subset of $\overline{\rho(G)}$, then there exist $\{g_v:v\in V\}\subseteq G$ such that for each $v, v'\in V$, we have
  \begin{align}
  	\label{eqn:distinct} v\neq v' &\implies g_{v}\neq g_{v'};\\
  	\label{eqn:vv'}
  v-v'\in U &\implies g_v-g_{v'}\in \rho^{-1}(U).
    \end{align} So let $V$ be a finite subset of $\overline{\rho(G)}$.  Let $S:=(V-V)\cap U$, and let $W$ be a neighborhood of $0$ in $K$ so that $S+W\subseteq U$.  Choose a neighborhood $W'$ of $0$ so that $W'-W'\subseteq W$.  Also choose $W'$ sufficiently small that the translates $v+W', v\in V$ are mutually disjoint.  For each $v\in V$, choose $g_v\in G$ so that $\rho(g_v)\in v+W'$; this is possible since $v\in \overline{\rho(G)}$.  Disjointness of the $v+W'$ implies (\ref{eqn:distinct}). To prove (\ref{eqn:vv'}), assume $v-v'\in U$.  Then
  \[\rho(g_v)-\rho(g_{v'})\in v+W' -(v'+W') = (v-v')+(W'-W')\subseteq v-v' + W \subseteq U,\]
  so $g_v-g_{v'}\in \rho^{-1}(U)$.  This proves the first assertion of the lemma.  Since independence ratio is invariant under isomorphism of graphs, we get that $\Cay(\rho^{-1}(U))$ has a finite subgraph with independence ratio less than $\delta$.  Lemma \ref{lem:CayleyCharacterization} then implies $\rho^{-1}(U)$ is a set of $\delta$-recurrence.
\end{proof}

\section{Approximate Hamming balls}\label{sec:BHballs}
\subsection{Independent sets}
\begin{definition}\label{def:Qindependent}
If $\Gamma$ is an abelian group, a set $S\subseteq \Gamma$ is \emph{independent} if the only solutions to the equation $n_1s_1+n_2s_2+\cdots+n_ks_k=0$ with all $s_{i}\in S$, $n_{i}\in \mathbb Z$ also satisfy $n_1s_1=n_2s_2=\cdots =n_ks_k=0$.
\end{definition}
Specializing to $\widehat{G}$ (with multiplicative notation), a finite set $\{\chi_{1},\dots,\chi_{k}\}\subseteq \widehat{G}$ is independent if the only $n_{1},\dots,n_{k}\in \mathbb Z$ satisfying $\chi_{1}(n_{1}g)\chi_{2}(n_{2}g)\cdots \chi_{k}(n_{k}g)=1$ for all $g\in G$ also satisfy $\chi_{1}(n_{1}g)=\chi_{2}(n_{2}g)=\cdots =\chi_{k}(n_{k}g)=1$ for all $g\in G$.

We continue to write $\mathcal S^1$ for the unit circle $\{z\in \mathbb C:|z|=1\}$ with the usual topology, and $Z_q$ for the $q^{\text{th}}$ roots of unity. We sometimes write $Z_{\infty}$ for $\mathcal S^1$, and will write $\mathcal S^d$ and $Z_q^d$ for the cartesian powers $(\mathcal S^1)^d$ and $(Z_q)^d$, respectively.  As before, $\widehat{G}_q$ will denote those $\chi\in \widehat{G}$ of order $q$.

\begin{lemma}[Kronecker's Lemma]\label{lem:Kronecker}
  Let $G$ be a countable abelian group, $q\in \mathbb N\cup\{\infty\}$, $r\in \mathbb N$, and let $\{\chi_{1},\dots,\chi_{r}\}\subseteq \widehat{G}_q$ be an independent collection of characters.  Let $\rho:G\to \mathcal S^r$ be given by $\rho(g)=(\chi_1(g),\dots,\chi_r(g))$.

  \begin{enumerate}
    \item[(i)] If $q=\infty$ then $\rho(G)$ is dense in $\mathcal S^r$.
    \item[(ii)]  If $q\in \mathbb N$, then $\rho(G)= Z_{q}^{r}\subseteq \mathcal S^{r}$.
        \end{enumerate}
\end{lemma}
\begin{proof}
  Parts (i) and (ii) are Theorem 5.1.3 of \cite{Rudin}.  \end{proof}

The following lemma provides finite approximations to the measures in Lemma \ref{lem:R4dot}.
\begin{lemma}\label{lem:IndependentSets}
  Let $G$ be a  countable abelian group.  For all $\varepsilon, \delta>0$ and $q\in \mathbb N\cup \{\infty\}$, there exists $N=N(\varepsilon,\delta,q)$ such that if $r\geq N$, $\{\chi_1,\dots,\chi_r\}$ is an independent subset of $\widehat{G}_q$, $S\subseteq G$ is measure expanding, $\phi:\{\chi_1,\dots, \chi_r\}\to Z_{q}$ is a function,  $F\subseteq G$ is finite, and
  \[S_{\phi,\varepsilon}:=\Bigl\{g\in S: \frac{1}{r}\sum_{j=1}^r|\phi(\chi_j)-\chi_j(g)| < \varepsilon \Bigr\}
  \]  there is a finite set $S_{\phi,\varepsilon}'\subseteq S_{\phi,\varepsilon}$ such that for every $h\in F$, $S_{\phi,\varepsilon}'+h$ is a set of $\delta$-recurrence.
\end{lemma}

The main argument in the proof of Proposition \ref{prop:KroneckerMeasures} will use the following corollary, so we prove it now. The remainder of this section is dedicated to the proof of Lemma \ref{lem:IndependentSets}.

\begin{corollary}\label{cor:Neighborhoods}
  For all $\varepsilon, \delta>0$, $q\in \mathbb N\cup \{\infty\}$, there exists $N=N(\varepsilon,\delta,q)$ such that if $r\geq N$, $\{\chi_1,\dots,\chi_r\}$ is an independent subset of $\widehat{G}_q$, $S\subseteq G$ is measure expanding, and $F\subseteq G$ is finite, then there are closed neighborhoods $U_{j}$ of $\chi_j$ such that for all functions $f:\widehat{G}\to Z_q$ which are constant on each $U_j$, setting
  \[
  Q_{f,\varepsilon}:=\Bigl\{g\in S:\frac{1}{r}\sum_{j=1}^r \sup_{\psi \in U_j} |f(\psi)-\psi(g)|<\varepsilon\Bigr\}
  \]
we have that for all $h\in F$, $Q_{f,\varepsilon}+h$ is a set of $\delta$-recurrence.
\end{corollary}

\begin{proof}
Fix $q\in \mathbb N\cup \{\infty\}$. Let $\delta,\varepsilon>0$ and let $F\subseteq G$ be finite. Choose a finite collection of functions $\mathcal F$ on $\{\chi_1,\dots,\chi_r\}$ such that every function $\phi:\{\chi_1,\dots,\chi_r\}\to Z_{q}$ can be $\varepsilon/3$-uniformly approximated by an element of $\mathcal F$.

For each $\phi\in \mathcal F$ let $S_{\phi,\varepsilon/3}'$ be the finite set given by Lemma \ref{lem:IndependentSets}, and let $S'=\bigcup_{\phi\in\mathcal F} S_{\phi,\varepsilon/3}'$, so that $S'$ is finite.  Since each function $e_g:\widehat{G}\to \mathcal S^1$ given by $e_g(\chi):=\chi(g)$ is continuous, there is a closed neighborhood $U_j$ of each $\chi_j$ such that $|\chi_j(g)-\psi(g)|<\varepsilon/3$ for all $\psi\in U_j$ and all $g\in S'$.

Now let $f:\widehat{G}\to \mathcal S^1$ be constant on each $U_j$.  We will prove that $S_{\tilde{f},\varepsilon/3}\subseteq Q_{f,\varepsilon}$ for some $\tilde{f}\in \mathcal F$.  To prove this, we first choose $\tilde{f}\in \mathcal F$ such that $|\tilde{f}(\chi_j)-f(\psi)|<\varepsilon/3$ for all $\psi\in U_j$, $j\leq r$.  For each $j$, $\psi\in U_j$, and $g\in S'$, we then have
\[|f(\psi)-\psi(g)|\leq |f(\psi)-\tilde{f}(\chi_j)|+|\tilde{f}(\chi_j)-\chi_j(g)|+|\chi_j(g)-\psi(g)|< \frac{\varepsilon}{3}+|\tilde{f}(\chi_j)-\chi_j(g)|+\frac{\varepsilon}{3},\]
so $\sup_{\psi\in U_j} |f(\psi)-\psi(g)|< |\tilde{f}(\chi_j)-\chi_j(g)| + 2\varepsilon/3.$  Thus if $g\in S_{\tilde{f},\varepsilon/3}'$, we have
\[
\frac{1}{r}\sum_{j=1}^r \sup_{\psi\in U_j} |f(\psi)-\psi(g)|< \frac{1}{r}\sum_{j=1}^r  |\tilde{f}(\chi_j)-\chi_j(g)|+2\varepsilon/3<\varepsilon.
\]
So $S_{\tilde{f},\varepsilon/3}'\subseteq Q_{f,\varepsilon}$.  By construction, $S_{\tilde{f},\varepsilon/3}'+h$ is a set of $\delta$-recurrence for every $\tilde{f}\in \mathcal F$, so $Q_{f,\varepsilon}+h$ is a set of $\delta$-recurrence, as well. \end{proof}

\subsection{Hamming balls}
For $q,r\in \mathbb N$, define $w:Z_q^r\to \mathbb N\cup \{0\}$ by
\[
w\bigl((x_1,\dots,x_r)\bigr):= |\{j\in \{1,\dots,r\}:x_j\neq 1\}|.
\]
For $y\in Z_q^r$, the Hamming ball of radius $k$ around $y$ is defined as
  \begin{equation}\label{eqn:HammingDef}
H(y;r):=\{x\in Z_{q}^{r}: w(x^{-1}y) \leq  k\}.
  \end{equation}
 In other words, the Hamming ball of radius $k$ around $y$ is the set of $r$-tuples of elements of $Z_{q}$ which differ from $y$ in at most $k$ coordinates.
 \begin{observation}\label{obs:HammingL1}
   If $x\in H(y;k)$, then $\sum_{j=1}^r |y_j-x_j|\leq 2k$, since $x_j,y_j\in \mathcal S^1$ for each $j$, and all but $k$ terms of the sum are $0$.
 \end{observation}

\subsection{Recurrence properties of Hamming balls}\label{sec:RecurrenceFinite}

Lemma 4.3 of \cite{GrPop} is the following.   Informally, it says that Hamming balls of moderate radius are sets of recurrence.

\begin{lemma}\label{lem:KleitmanRecurrence}
For all $\varepsilon,\delta>0$ and all $q\in \mathbb N$,  there exists $N=N(\varepsilon,\delta, q)$ such that: if $r\geq N$ and $A\subseteq Z_{q}^{r}$ has $|A| > \delta |Z_{q}^{r}|$, then for all $y\in Z_{q}^{r}$, there are $a_{1},a_{2}\in A$ such that $a_{2}^{-1}a_1\in H(y,\varepsilon r)$.
\end{lemma}
In the terminology of \S\ref{sec:CayleyGraphs}, the conclusion says $\Cay(H(y,\varepsilon r))$ has independence ratio at most $\delta$, meaning $H(y,\varepsilon r)$ is a set of $\delta$-recurrence in $Z_q^r$.

Facts similar to Lemma \ref{lem:KleitmanRecurrence} have traditionally been used to construct examples like those in Theorem \ref{thm:Main0}, cf.~\cite{Forrest, ForrestThesis,McCutcheonAlexandria,Wolf}.

Most of the proof Lemma \ref{lem:IndependentSets} will be done in the next lemma.

\begin{lemma}\label{lem:Prelim}
  Let $G$ be a  countable abelian group.  For all $\varepsilon, \delta>0$ and $q\in \mathbb N\cup \{\infty\}$, there exists $N=N(\varepsilon,\delta,q)$ such that if $r\geq N$, $\{\chi_1,\dots,\chi_r\}\subseteq \widehat{G}_q$ is independent, and $\phi:\{\chi_1,\dots,\chi_r\}\to Z_{q}$, then with
  \[P_{\phi}:=\Bigl\{g\in G: \frac{1}{r}\sum_{j=1}^r|\phi(\chi_j)-\chi_j(g)| < \varepsilon \Bigr\}
  \]  there is set of $\delta$-recurrence $E$ and a Bohr neighborhood $B$ of $0$ in $G$ such that $E+B\subseteq P_{\phi}$.
\end{lemma}

\begin{proof}
     We first prove the lemma for $q=\infty$. In this case,      choose $q'>4\pi/\varepsilon$, and make $N\geq N(\varepsilon/8,\delta,q')$ in Lemma \ref{lem:KleitmanRecurrence}. Assume $r\geq N$, and fix an independent set of characters $\{\chi_1,\dots,\chi_r\}$ of infinite order.

  Consider the homomorphism $\rho:G\to \mathcal S^r$ given by $\rho(g) = (\chi_1(g),\dots,\chi_r(g))$. Then  $\overline{\rho(G)}=\mathcal S^r$, by Lemma \ref{lem:Kronecker}. For each $j$, choose $y_j\in Z_{q'}$ so that $|\phi(\chi_j)-y_j|<\varepsilon/2$, as our choice of $q'$ permits.  Let 
  \[
  U: = \Bigl\{(z_{1},\dots,z_{r})\in \mathcal S^{r}: \frac{1}{r}\sum_{j=1}^r |y_j-z_j|<\varepsilon/4\Bigr\}.
  \]
  Then $U$ is an open subset of $\mathcal S^r$. By Observation \ref{obs:HammingL1}, $U$ contains $H(y,\varepsilon r/8)$ in $Z_{q'}^r$.  Our choice of $r$ makes $H(y,\varepsilon r/8)$ a set of $\delta$-recurrence in $Z_{q'}^r$, so $\rho^{-1}(U)$ is a set of $\delta$-recurrence in $G$, by Lemma \ref{lem:CopyCayley}.

  Let $B$ be the Bohr neighborhood $\{h\in G: \max_{j\leq r} |\chi_j(h)-1|<\varepsilon/4\}$.  We claim that $\rho^{-1}(U)+B\subseteq P_{\phi}$.  To see this, note that if $g\in \rho^{-1}(U)$ and $h\in B$, we have $\frac{1}{r}\sum_{j=1}^r |y_j-\chi_j(g)|<\varepsilon/4$ and $|\chi_j(g)-\chi_j(g+h)|=|1-\chi_j(h)|<\varepsilon/4$ for each $j$.  Then
  \begin{align*}
    \frac{1}{r}\sum_{j=1}^r |\phi(\chi_j)- \chi_j(g+h)|      &\leq \frac{1}{r}\sum_{j=1}^r |\phi(\chi_j)-y_j|+|y_j-\chi_j(g)| + |\chi_j(g)-\chi_j(g+h)|\\
    & < \varepsilon/2 + \varepsilon/4 + \varepsilon/4.
  \end{align*}
So $\frac{1}{r}\sum_{j=1}^r |\phi(\chi_j)- \chi_j(g+h)|  <\varepsilon$, meaning $g+h\in P_{\phi}$.  Thus $\rho^{-1}(U)+B\subseteq P_{\phi}$.  This completes the proof when $q=\infty$.

When $q$ is finite the lemma can be proved by repeating the argument above, taking $q'=q$ instead of $q'>4\pi/\varepsilon$, since $\phi$ only takes values in $Z_q$. \end{proof}

\begin{proof}[Proof of Lemma \ref{lem:IndependentSets}]
Fix $\varepsilon,\delta>0$ and $q\in \mathbb N\cup \{\infty\}$.   Let $N>N(\varepsilon,\delta/2,q)$ in Lemma \ref{lem:Prelim}.  Fix $\chi_j$ and $\phi$ as in the hypothesis of Lemma \ref{lem:IndependentSets}.  Note that $S_{\phi,\varepsilon}=S\cap P_{\phi}$, where $P_{\phi}$ is defined in Lemma \ref{lem:Prelim}.  By Lemma \ref{lem:UniformRecurrence}, it suffices to prove that every translate of $S\cap P_{\phi}$ is a set of $\delta/2$-recurrence.  To this end, note that for each $h\in G$, $P_{\phi}+h=P_{\phi'}$, where $\phi'(\chi_j)=\chi_j(h)\phi(\chi_j)$, since $|\phi(\chi_j)-\chi_j(g-h)|=|\chi_j(h)\phi(\chi_j)-\chi_j(g)|$.

Thus, by Lemma \ref{lem:Prelim}, we get that for each $h$, $P_{\phi}+h$ contains a set of the form $E+B$, where $E$ is a set of $\delta/2$-recurrence and $B$ is a Bohr neighborhood of $0$ in $G$.  Then Lemma \ref{lem:Aura} implies $(S+h)\cap (P_{\phi}+h)$ is a set of $\delta/2$-recurrence for each $h\in G$.  This means every translate of $S\cap P_{\phi}$ is a set of $\delta/2$-recurrence.
\end{proof}

\section{Independent sets in compact groups}\label{sec:IndependentSets}

 Lemma \ref{lem:ExistIndependent} below will be used to prove Proposition \ref{prop:KroneckerMeasures}.
\begin{lemma}\label{lem:EssentialOrder}
  Let $q_1,\dots,q_r\in \mathbb N$, $q_i\geq 2$ for each $i$. Let $q=\operatorname{lcm}(q_1,\dots,q_r)$.  Then the set of elements of order $q$ in $Z:=D_{q_1}\times \cdots \times D_{q_r}$ is dense and open in $Z$.
\end{lemma}

\begin{proof}
  Writing an element of $Z$ as $z=(z_{i,j})_{i\in \mathbb N,j\in [r]}$, it is enough to observe that for each $n$, the set $U_n:=\{z: z_{i,j}=1 \text{ if } i= n\}$ is open and consists of elements of order $q$, while $\bigcup_{n\in \mathbb N} U_n$ is dense in $Z$.
\end{proof}

\begin{lemma}\label{lem:ExistIndependent}
Let $Z$ be an infinite compact abelian group and let $V_{1},\dots, V_{k}$ be mutually disjoint nonempty open subsets of $Z$.

If $Z$ has infinite exponent, then there exist $z_j\in V_{j}$ with infinite order such that $\{z_{1},\dots, z_{k}\}$ is independent.

If $Z=D_{q_1}\times \cdots \times D_{q_r}$  then there exist $z_{j}\in V_{j}$ of order $q=\operatorname{lcm}(q_1,\dots,q_k)$, such that $\{z_{1},\dots, z_{k}\}$ is independent.
\end{lemma}

\begin{proof}  The infinite exponent case is an immediate consequence of Lemma 5.2.3(a) of \cite{Rudin}. To prove the remaining case we follow the proof of \cite[Lemma 5.2.3(b)]{Rudin}.  So assume $Z=D_{q_1}\times \cdots \times D_{q_r}$.  For $y\in Z$, $m\in \mathbb N$, let $E_{m,y}:=\{z\in Z: mz= y\}$.  Note that if $z, z'\in E_{m,y}$ then $z-z'\in E_{m,0}$, and that $E_{m,0}$ is a closed subgroup of $Z$.  Thus each $E_{m,y}$ is contained in a coset of $E_{m,0}$.  Furthermore, if $m$ is not divisible by $q$, then $E_{m,0}$ has empty interior, by Lemma \ref{lem:EssentialOrder}. For such $m$ and every $y\in Z$,  $E_{m,y}$ has empty interior, as well. Now
  \[W:= Z\setminus \bigcup_{m\geq 1,q\nmid m} E_{m,0}\] dense and open (as there are only finitely many distinct sets $E_{m,0}$), and $V_{j}':=V_{j}\cap W$, $1\leq j\leq r$ defines a collection of mutually disjoint open sets.

  Choose $z_{1}\in V_{1}'$. Note that $nz_{1}=0$ implies $q$ divides $n$, since $V_{1}'\subseteq W$.   Thus $\{z_{1}\}$ is independent and its one element has order $q$. Proceeding by induction, suppose $z_{1},\dots, z_{j}$, $z_{i}\in V_{i}'$, $j<k$ are chosen so that $\{z_{1},\dots, z_{j}\}$ is an independent set of elements of order $q$.  We will find $z_{j+1}\in V_{j+1}'$ so that $\{z_{1},\dots, z_{j+1}\}$ is independent and $z_{j+1}$ has order $q$.  Note that the group $H$ generated by $\{z_{1},\dots, z_{j}\}$ is finite, as $Z$ has finite exponent.  Then $U:=Z\setminus \bigcup\{E_{m,y}: m\in \mathbb N, q\nmid m, y\in H\}$ is dense, so Lemma \ref{lem:EssentialOrder} provides a $z_{j+1}\in U\cap V_{j+1}'$ having order $q$.  Now we prove that $\{z_{1},\dots, z_{j+1}\}$ is independent, so assume $n_{1},\dots,n_{j+1}\in \mathbb Z$, and that $n_{1}z_{1}+\cdots + n_{j+1}z_{j+1}=0$. Since $z_{j+1}\in U$,  the equation $-n_{j+1}z_{j+1}= n_{1}z_{1}+\cdots+n_{j}z_{j}$ implies $q$ divides $n_{j+1}$. For such $n_i$, it follows that $n_{1}z_{1}+\cdots+n_{j}z_{j}=0$, so the independence of $\{z_{1},\dots, z_{j}\}$ implies $q$ divides each of the $n_{i}$.  We have shown that if $n_{1}z_{1}+\cdots+n_{j+1}z_{j+1}=0$, then $q$ divides $n_{i}$ for each $i$, meaning $\{z_1,\dots,z_{j+1}\}$ is independent.
 \end{proof}

\section{Finding independence via continuous measures}\label{sec:Measures}

\subsection{Cantor sets}\label{sec:Cantor}

Suppose $\Gamma$ is a non-discrete locally compact metrizable group, and let $L\subseteq \Gamma$ be a perfect subset of $\Gamma$ (so that $L$ is compact and has no isolated points).  Furthermore assume $d_{\Gamma}$ is a metric on $\Gamma$ generating its topology.  The following general procedure will construct a subset of $L$ homeomorphic to the Cantor set.

 Let $(b_{k})_{k\in \mathbb N}$ be a sequence of integers, $b_{k}\geq 2$ for each $k$.   Let $N_{1}=b_{1}$, and $N_{k+1}=b_{k}N_{k}$.   Let $\mathcal U^{(k)}=\{U_{1}^{(k)},\dots, U_{N_{k}}^{(k)}\}$ be a collection of $N_{k}$ mutually disjoint compact neighborhoods in $L$, so that

\begin{enumerate}
  \item[(C1)] for each $k$, $j$, $U_{j}^{(k+1)}\subseteq U_{l}^{(k)}$ for some $l$;

\item[(C2)] each $U\in \mathcal U^{(k)}$ contains exactly $b_{k}$ elements of $\mathcal U^{(k+1)}$;

\item[(C3)]$\lim_{k\to \infty} \sup_{U\in \mathcal U^{(k)}} \operatorname{diam}(U)=0$.
 \end{enumerate}
Then $K:=\bigcap_{k=1}^{\infty} \bigcup \mathcal U^{(k)}$ is homeomorphic to the Cantor set, as it is compact, totally disconnected, and has no isolated points.
\subsection{Cantor measure}\label{sec:CantorMeasure} With $K$ as in \S\ref{sec:Cantor}, we construct a probability measure $\sigma_{\mathcal U}$ on $K$ associated to the sequence $\mathcal U^{(k)}$: for each $k$ and each $U\in \mathcal U^{(k)}$, choose a point $x_{U}\in U$. Consider the measures $\sigma_{k}:=\frac{1}{|\mathcal U^{(k)}|}\sum_{U\in \mathcal U^{(k)}} \delta_{x_{U}}$, where $\delta_{x_{U}}$ is the unit point mass at $x_{U}$.   Then $\lim_{k\to \infty} \sigma_{k}$ exists in the weak$^{*}$ topology on $C(\Gamma)$; let $\sigma_{\mathcal U}$ be the limit.  The notation is appropriate as $\sigma_{\mathcal U}$ does not depend on the exact choice of $x_{U}\in U$.  We call $\sigma_{\mathcal U}$ the \emph{Cantor measure on $K$ induced by $\mathcal U$}.  Note that $\sigma_{\mathcal U}$ is continuous: if $x\in K$ and $k\in \mathbb N$, we have $x\in U$ for some $U\in \mathcal U^{(k)}$.  For each $n\geq k$, we have $\sigma_{n}(U)=\frac{1}{|\mathcal U^{(k)}|}$, so $\sigma_{\mathcal U}(\{x\}) \leq \frac{1}{|\mathcal U^{(k)}|}$ for every $k$, meaning $\sigma_{\mathcal U}(\{x\})=0$.

\subsection{Saeki measures}\label{sec:SaekiGeneral} We recall Definition \ref{def:SaekiMeasure}:   if $G$ is a countable abelian group, $S\subseteq G$, and $\mathcal P$ is a collection of subsets of $G$, we say that a positive Borel measure $\sigma$ on $\widehat{G}$ is a $\mathcal P$-\emph{Saeki measure} if for all $\varepsilon>0$, and every continuous $f:\widehat{G} \to \mathcal S^{1}$,
  \begin{equation}\label{eqn:SMeasureDef}
  \{g\in G: \textstyle\int|f(\chi)-\chi(g)| \,d\sigma(\chi)<\varepsilon \}\in \mathcal P.
  \end{equation}
   If $q\in \mathbb N$, we say that $\sigma$ is a $\mathcal P^{(q)}$-Saeki measure if for all $\varepsilon>0$ and every measurable $f:\widehat G \to Z_{q}$, the inclusion (\ref{eqn:SMeasureDef}) is satisfied. (Recall that $Z_q$ denotes the $q^\text{th}$ roots of unity.)
   
We now introduce some notation for the proof of Lemma \ref{lem:AbstractSaeki}.

Fix a sequence of integers $b_{k}$ with $b_{k}\geq 2$ for all $k$, and a sequence $\mathcal U^{(k)}$ of collections of compact neighborhoods satisfying conditions (C1)-(C3) listed in \S\ref{sec:Cantor}.  For $f:\widehat{G}\to \mathbb C$, let
\begin{equation}\label{eqn:wkeps}
M_{k}(f):= \frac{1}{|\mathcal U^{(k)}|}\sum_{U\in \mathcal U^{(k)}} \sup_{\chi \in U}|f(\chi)|.
\end{equation}
\begin{observation}\label{obs:Decreasing}Let $f:\widehat{G}\to \mathbb C$ be continuous, and let $\sigma$ be the Cantor measure associated to $\mathcal U$. Then
\begin{enumerate}	
\item[(i)]  $M_{k}(f)$ is a decreasing function of $k$;
\item[(ii)]	$\int |f|\, d\sigma \leq \lim_{k\to\infty}  M_{k}(f)$;
\item[(iii)] If $f':\widehat{G}\to \mathbb C$, then $M_{k}(f+f')\leq M_{k}(f)+M_{k}(f')$.
    \end{enumerate}
\end{observation}
To prove part (i), write $\frac{1}{|\mathcal U^{(k+1)}|}\sum_{U\in \mathcal{U}^{(k+1)}}$ as $\frac{1}{|\mathcal U^{(k)}|}\sum_{U\in \mathcal{U}^{(k)}}\frac{1}{b_k}\sum_{V\in \mathcal U^{(k+1)}, V\subseteq U}$, and note that for each $k$, $\frac{1}{b_k}\sum_{V\in \mathcal U^{(k+1)}, V\subseteq U} \sup_{\chi\in V} |f(\chi)|\leq \sup_{\chi \in U}|f(\chi)|$.

Part (ii) follows from the obvious bound $\int |f|\, d\sigma_k \leq M_k(f)$. Part (iii) is just the pointwise triangle inequality $|f+g|\leq |f|+|g|$.

For $g\in G$, define  $e_g:\widehat{G}\to \mathcal S^1$ by $e_g(\chi):=\chi(g)$. Then $e_g$ is continuous.

For $\varepsilon>0$ and $f:\widehat{G}\to \mathcal S^{1}$, let
   \begin{equation}\label{eqn:DefQ}
  Q(k,f,\varepsilon):=   \{g\in G: M_{k}(f-e_g)<\varepsilon\}.
\end{equation}

A collection $\mathcal P$ of sets is \emph{upward closed} if for all $A, B$ with $A\subseteq B$, $A \in \mathcal P$ implies $B\in \mathcal P$.

\begin{lemma}\label{lem:AbstractSaeki}  Let $\mathcal P$, $\mathcal P_{1}$, $\mathcal P_{2}$, $\dots$ be upward closed collections of subsets of $G$ satisfying
\begin{enumerate}
\item[(P1)] if $E\in \mathcal P_{k}$ for infinitely many $k$, then $E\in \mathcal P$.  \end{enumerate}
For each $k\in \mathbb N$, let $b_{k}\in \mathbb N$, $b_{k}\geq 2$, and let $\mathcal U^{(k)}$ be a collection of compact neighborhoods in $\widehat{G}$ satisfying (C1)-(C3).  Assume 
\begin{enumerate}\item[(S1)]     for all $\varepsilon>0$ and continuous $f:\widehat{G}\to \mathcal S^1$,  $Q(k,f,\varepsilon)\in \mathcal P_{k}$ for infinitely many $k$, 
\end{enumerate}
with $Q(k,f,\varepsilon)$ as defined in (\ref{eqn:DefQ}).  Let $K:=\bigcap_{k=1}^{\infty} \bigcup \mathcal U^{(k)}$. Then the Cantor measure $\sigma:=\sigma_{\mathcal U}$ on $K$ is a continuous $\mathcal P$-Saeki measure.

If instead of (S1) we have
\begin{enumerate}
 \item[(S1$_{q}$)]  for all $\varepsilon>0$ and continuous $f: \widehat{G}\to Z_q$,  $Q(k,f,\varepsilon)\in \mathcal P_{k}$ for infinitely many $k$,
\end{enumerate}
 then $\sigma$ is a continuous $\mathcal P^{(q)}$-Saeki measure.
\end{lemma}

\begin{proof}
By definition, $\sigma$ is the weak$^{*}$ limit of the measures $\sigma_{k}:=\frac{1}{|\mathcal U^{(k)}|}\sum_{U\in \mathcal U^{(k)}} \delta_{x_{U}}$, where each $x_{U}\in U$. Continuity of $\sigma$ was proved in \S\ref{sec:CantorMeasure}.

  Assuming (S1) holds, we will prove that $\sigma$ is a $\mathcal P$-Saeki measure.  It suffices to prove that if $f:\widehat{G}\to \mathcal S^{1}$ is continuous and $\varepsilon>0$, then
    \begin{equation}\label{eqn:IsInP}
      S_{f,\varepsilon}:=\{g\in S: \textstyle \int |f-e_g| \, d\sigma <\varepsilon\}\in \mathcal P,
    \end{equation}so let $f:\widehat{G}\to\mathcal S^1$ be  continuous and let $\varepsilon>0$.  We will prove  (\ref{eqn:IsInP}) by showing that $S_{f,\varepsilon}$ is in $\mathcal P_{k}$ for infinitely many $k$. By (P1) this will imply $S_{f,\varepsilon}\in \mathcal P$.  To prove $S_{f,\varepsilon}\in \mathcal P_{k}$, we claim that for each $k$,
    \begin{equation}\label{eqn:SfepsInQ}
    Q(k,f,\varepsilon)\subseteq S_{f,\varepsilon}
    \end{equation}
    To see this, note that $\int |f-e_g|\, d\sigma\leq M_k(f-e_g)$ by Observation \ref{obs:Decreasing} (ii), while $g\in Q(k,f,\varepsilon)$ if and only if $M_k(f-e_g)<\varepsilon$.  Now (\ref{eqn:SfepsInQ}), the upward-closedness of $\mathcal P_{k}$, and the hypothesis (S1) together imply $S_{f,\varepsilon}\in \mathcal P_{k}$ for infinitely many $k$.  Our assumption (P1) then implies $S_{f,\varepsilon}\in \mathcal P$, so we have proved that $\sigma$ is a $\mathcal P$-Saeki measure.

The proof that $\sigma$ is a $\mathcal P^{(q)}$-Saeki measure under assumption (S1$_{q}$) is nearly identical: we may assume that $\varepsilon>0$ and $f:\widehat{G}\to Z_{q}$ rather than $f:\widehat{G}\to \mathcal S^{1}$, then prove that $S_{f,\varepsilon}\in \mathcal P$ in exactly the same manner as above.  \end{proof}

\subsection{Construction of Saeki measures for recurrence}\label{sec:Construction}

\begin{proof}[Proof of Lemma \ref{lem:R4dot}] We first prove Part 1.  Assume $G$ has infinite exponent and enumerate $G$ as $\{g_{1},g_{2},\dots\}$.  Let $W\subseteq \widehat{G}$ be open.  Let $S\subseteq G$ be measure expanding, and let $\mathcal R_{k}S$ be the collection of sets $E\subseteq S$ such that $E+g_{j}$ is a set of $1/k$-recurrence for all $j\leq k$.     We will apply Lemma \ref{lem:AbstractSaeki} with $\mathcal R^{\bullet}S$ in place of $\mathcal P$, and $\mathcal R_{k}S$ in place of $\mathcal P_{k}$.

  Let $N_{0}=1$, $b_{0}=1$, and for each $m\in \mathbb N$, let $N_{n}> N(1/n,1/n,\infty)$ in Corollary \ref{cor:Neighborhoods}.  For convenience, make $N_{n}$ have the recursive form $N_{n+1}=b_{n}N_{n}$ for some sequence of natural numbers $b_{n}\geq 2$.  Write $[n]$ for the interval of integers $\{1,\dots, n\}$.

   We define a sequence of collections $\mathcal U^{(k)}$ of closed neighborhoods in $\widehat{G}$.  Let $\mathcal U^{(0)}:=\{U^{(0)}_{1}\}$, where $U^{(0)}_{1}$ is an arbitrary closed neighborhood contained in $W$.   Suppose $\mathcal U^{(k)}$ is defined for $k=0,\dots, n-1$.  Apply Lemma \ref{lem:ExistIndependent} to find, for each $U\in \mathcal U^{(n-1)}$, characters   $\chi_{U,1}, \chi_{U,2}, \dots, \chi_{U,b_{n-1}}\in U$ such that  $K_{n}:=\{\chi_{U,j}: U\in \mathcal U^{(n-1)}, j\in [b_{n-1}]\}$ is independent, and each $\chi \in K_n$ has infinite order.  Note that $|K_{n}|= b_{n-1} N_{n-1}= N_{n}$.

  By Corollary \ref{cor:Neighborhoods}, choose closed neighborhoods $U_{\chi}$ of each $\chi \in K_n$, such that for each $f:\widehat{G}\to \mathcal S^1$ which is constant on every $U_{\chi}$, setting
  \[
  S_{f,n}:=\Bigl\{g\in S: \frac{1}{|K_n|}\sum_{\chi\in K_n}\sup_{\psi\in U_\chi}|f(\psi)-e_g(\psi)|<\frac{1}{n}\Bigr\}
  \]
   we have that $S_{f,n}+g_j$ is a set of $\frac{1}{n}$-recurrence for every $j\leq n$.  We also insist that the $U_\chi$ are mutually disjoint and have diameter at most $1/n$. Then let $\mathcal U^{(n)}=\{U_{\chi}:\chi \in K_{n}\}$.  This completes the inductive definition of $\mathcal U^{(n)}$.  As in \S\ref{sec:Cantor}, we set $K=\bigcap_{k=1}^\infty \bigcup \mathcal U^{(k)}$, and let $\sigma$ be the associated Cantor measure.  We will prove that $\sigma$ is an $\mathcal R^{\bullet}S$-Saeki measure.

Recall the definition of $M_{n}$ from (\ref{eqn:wkeps}) and  $Q(n,f,\frac{1}{n}):=\{g\in S: M_n(f-e_g)<\frac{1}{n}\}$.  Note that if $f:\widehat{G}\to \mathcal S^{1}$ is constant on every $U\in \mathcal U^{(n)}$, then $S_{f,n}= Q(n,f,\tfrac{1}{n})$.  So our choice of $S_{f,n}$ implies
\begin{equation}\label{eqn:QnIsInRn}
  Q(n,f,\tfrac{1}{n}) \in \mathcal R_{n}S.
\end{equation}

We will show that the hypothesis of Lemma \ref{lem:AbstractSaeki} is satisfied with $\mathcal R^{\bullet}S$ in place of $\mathcal P$ and $\mathcal R_{n}S$ in place of $\mathcal P_{n}$.

Here condition (P1) simply means that if, for infinitely many $n$, $E+g_{j}$ is a set of $\frac{1}{n}$-recurrence for every $j\leq n$, then every translate of $E$ is a set of recurrence; this implication follows immediately the relevant definitions.

(C1), (C2), and (C3) are easily verified from the definition of $\mathcal U^{(n)}$.

We now verify (S1).  Let $f:\widehat{G}\to \mathcal S^1$ be continuous and $\varepsilon>0$.  By our construction of $K$, there is an $N\in \mathbb N$ and a continuous $\tilde{f}:\widehat{G}\to Z_N$ which is constant on each $U\in \mathcal U^{(N)}$, such that $\sup_{\chi\in \bigcup \mathcal U^{(N)}} |f(\chi)-\tilde{f}(\chi)|<\varepsilon/2$.  We will prove that
\begin{equation}\label{eqn:Monotone}
  Q(N,\tilde{f},\varepsilon/2) \subseteq Q(n,f,\varepsilon) \qquad \text{for each } n\geq N.
\end{equation}
Fixing $n\geq N$, $g\in Q(N,\tilde{f},\varepsilon/2)$, we have \begin{align*}
M_{n}(f-e_g)&=M_{n}(f-\tilde{f}+\tilde{f}-e_g)\\
&\leq M_{n}(f-\tilde{f})+M_{n}(\tilde{f}-e_g)  && \text{by Observation \ref{obs:Decreasing} (iii)}\\
&\leq \frac{\varepsilon}{2} + M_{N}(\tilde{f}-e_g) && \text{by Observation \ref{obs:Decreasing} (i)}\\
&\leq \frac{\varepsilon}{2} +\frac{\varepsilon}{2} && \text{since } g\in Q(N,\tilde{f},\varepsilon/2).
\end{align*}
 The containment (\ref{eqn:Monotone}) and the inclusion (\ref{eqn:QnIsInRn}) now imply that for all $\varepsilon>0$, $Q(n,f,\varepsilon)\in \mathcal R_nS$ holds for infinitely many $n$, so (S1) is satisfied.

Having verified the hypotheses of Lemma \ref{lem:AbstractSaeki}, we conclude that  $K:=\bigcap_{n=1}^{\infty} \bigcup \mathcal U^{(n)}$ is a compact set supporting a continuous $\mathcal{R}^{\bullet}S$-Saeki measure.  This completes the proof of Part 1.

To prove Part 2, assume $G = \Sigma_{q_1}\oplus \cdots \oplus \Sigma_{q_r}$, and let $q$ be the exponent of $G$.  The construction of $K_n$, $\mathcal U^{(n)}$, and $S_{n,f}'$ is the same, except that when applying Lemma \ref{lem:ExistIndependent}, we choose the $\chi_{U,j}$ to each have order $q$.
\end{proof}

\section{The exceptional groups}\label{sec:arbitrary}

We prove Theorem \ref{thm:MainFinite} and Corollary \ref{cor:Main} after some preliminary lemmas.

\subsection{Rigidity sequences in products}

\begin{lemma}\label{lem:ExtendGaussian} Let $F$ be a finite abelian group and $H$ a countably infinite abelian group. Let $G=F\times H$, let $(g_n)=(0,h_n)$ be a sequence of elements of $\{0\}\times H$, and suppose that there is a free weak mixing $H$-system $(X,\mu,T)$ for which $(h_n)$ is a rigidity sequence.  Then  there is a free weak mixing $G$-system $(Y,\nu,S)$ for which $(g_n)$ is a rigidity sequence.
\end{lemma}

\begin{proof}
  Assuming $(h_n)$ is a rigidity sequence for a free weak mixing $H$-system, by Lemma \ref{lem:FourierCharacterizeFree} there is a continuous probability measure $\sigma$ on $\widehat{H}$ such that $\lim_{n\to\infty} \hat{\sigma}(h_n)=1$ and the support of $\sigma$ generates a dense subgroup of $\widehat{H}$.  We identify $\widehat{F\times H}$ with $\widehat{F}\times \widehat{H}$, and define a measure $\eta:=m\times \sigma$ on $\widehat{F\times H}$, where $m$ is normalized counting measure on $\widehat{F}$.

  If $(\psi,\chi)\in \widehat{F}\times \widehat{H}$, we have $(\psi,\chi)(a,h)=\psi(a)\chi(h)$.  For $h\in H$, we then have $\hat{\eta}(0,h)=\int \psi(0)\chi(h)\, dm(\psi)\, d\sigma(\chi)= \hat{\sigma}(h)$.  Thus $\lim_{n\to\infty} \hat{\eta}(0,h_n)=\lim_{n\to\infty} \hat{\sigma}(h_n)=1$.  It is easy to check that $\eta$ is continuous and its support generates a dense subgroup of $\widehat{F}\times \widehat{H}$.  By Lemma \ref{lem:FourierCharacterizeFree}, there is a free weak mixing $G$-system for which $(g_n)$ is a rigidity sequence.
\end{proof}

\subsection{Measure expanding sets in subgroups}\label{sec:Subgroups}

\begin{lemma}\label{lem:FiniteIndex}
  Let $G$ be a countable abelian group, $H\subseteq G$ a finite index subgroup, and $S\subseteq G$. Then

  \begin{enumerate}
    \item[1.]  $S$ is a set of recurrence if and only if $S\cap H$ is a set of recurrence.

    \item[2.]  $S$ is a set of strong recurrence if and only if $S\cap H$ is a set of strong recurrence.

    \item[3.]  $S$ is measure expanding if and only if for all $g\in G$, the set $(S-g)\cap H$ is a set of recurrence.
\end{enumerate}

\end{lemma}

\begin{proof} Let $(X,\mu,T)$ be a $G$-system and let $D\subseteq X$ have $\mu(D)>0$. Consider the group rotation $G$-system $(Z,m,T_{\rho})$ where $Z:=G/H$ is finite, $m$ is normalized counting measure on $Z$, and $\rho:G\to H$.  Let $E:=\{0_{Z}\}$ and $C:=D\times E$. Consider the product system $ (X\times Z, \mu\times m, T\times T_{\rho})$, and observe that $\mu\times m(C)>0$, while $C\cap (T\times T_{\rho})^{g}C=\varnothing$ unless $g\in H$.

 To prove Part 1, assume  $S$ is a set of recurrence.   Then there is a $g\in S$ such that $\mu\times m(C\cap (T\times T_{\rho})^{g}C)>0$.  For such $g$, we have $\mu(D\cap T^{g}D)>0$, and as mentioned above we must have $g\in H$, as well.  Since $(X,\mu,T)$ is an arbitrary $G$-system and $D$ is an arbitrary set of positive measure, we have shown that $S\cap H$ is a set of recurrence.

 The reverse implication in Part 1 is obvious.

  Part 2.  Assume $S$ is a set of strong recurrence.  Since $S$ is a set of strong recurrence, there is a $c>0$ such that $\mu\times m(C\cap (T\times T_{\rho})^{g}C)>c$ for infinitely many $g\in S$, and for such $g$ we have $\mu(D\cap T^{g}D)>c$.   These $g$ belong to $H$, as well, since $C\cap (T\times T_{\rho})^{g}C$ is empty otherwise.  Thus $S\cap H$ is a set of strong recurrence.

  Again the reverse implication in Part 2 is obvious.

  Part 3 follows immediately from Part 1. \end{proof}

\begin{proof}[Proof of Theorem \ref{thm:MainFinite}]
 Let $r, q_1,\dots,q_r\in \mathbb N$, $q_{i}\geq 2$, and let $F$ be a finite abelian group.  Let $G=F\oplus (\Sigma_{q_1} \oplus\cdots \oplus \Sigma_{q_r})$, so that $H:= \{0\}\oplus \Sigma_{q_1} \oplus\cdots \oplus \Sigma_{q_r}$ is a finite index subgroup of $G$.

 Assume $S\subseteq G$ is such that for all $h\in H$, $S+h$ is a set of recurrence.  Then $S\cap H$ has this property as well, by Lemma \ref{lem:FiniteIndex}.  Identifying $S\cap H$ with a subset of $\Sigma_{q_1} \oplus\cdots \oplus \Sigma_{q_r}$, we see that $S\cap H$ is measure expanding (for measure preserving $H$-systems).  Theorem \ref{thm:Main0} then provides a sequence $(s_n)$ of elements of $S\cap H$ such that $S':=\{s_n:n\in \mathbb N\}$ is measure expanding for $H$-systems, and for all $h\in H$, $(s_n+h)_{n\in \mathbb N}$ is a rigidity sequence for a free weak mixing $H$-system.   For each $h\in H$, Lemma \ref{lem:ExtendGaussian} then provides a free weak mixing $G$-system for which $(s_{n}+h)_{n\in \mathbb N}$ is a rigidity sequence. It is easy to verify that for all $h\in H$, $S'+h$ is a set of recurrence for $G$-systems.
\end{proof}

\subsection{Proof of Corollary \ref{cor:Main}}

We will need the following lemmas.

\begin{lemma}\label{lem:Translate}
  Suppose $S$ is a set of rigidity for the $G$-system $(X,\mu,T)$, $h\in G$, and $T^{h}$ is a nontrivial measure preserving transformation of $X$.  Then $S+h$ is not a set of strong recurrence.
\end{lemma}

\begin{proof}
  Enumerate $S$ as $(s_{n})_{n\in \mathbb N}$, a rigidity sequence for $(X,\mu, T)$.  Since $T^{h}$ is nontrivial, there is a set $D\subseteq X$ with $\mu(D)>0$ such that $T^{-h} D\cap  D=\varnothing.$  Then
  \[\lim_{n\to \infty} \mu(D\cap T^{s_{n}+h}D)=\lim_{n\to\infty} \mu(T^{-h}D\cap T^{s_{n}}D)=\mu(T^{-h}D\cap D)=0.\]  This shows that $S+h$ is not a set of strong recurrence.
\end{proof}

\begin{lemma}\label{lem:Divisibility}
	Let $S_{1}, S_{2}\subseteq G$.  If $S_{1}$ and $S_{2}$ are both not sets of strong recurrence, then $S_{1}\cup S_{2}$ is not a set of strong recurrence.
\end{lemma}

\begin{proof}  For $i=1,2$, let $(X_{i},\mu_{i},T_{i})$ be $G$-systems with sets $D_{i}\subseteq X_{i}$ having $\mu_{i}(D_{i})>0$, such that for all $c>0$, $E_{i}:=\{s\in S_{i}:\mu_{i}(D_{i}\cap T_{i}^{s}D_{i})>c\}$ is finite.  Let $(X,\mu,T)$ be the product system $(X_{1}\times X_{2}, \mu_{1}\times \mu_{2}, T_{1}\times T_{2})$, and let $D=D_{1}\times D_{2}$.  Then $\mu(D)>0$, and  for all $c>0$, $\{s\in S_{1}\cup S_{2}: \mu(D\cap T^{s} D)>c\}$ is contained in $E_{1}\cup E_{2}$, and is therefore finite. Thus $S_{1}\cup S_{2}$ is not a set of strong recurrence. \end{proof}

\begin{proof}[Proof of Corollary \ref{cor:Main}.]
When $G$ has infinite exponent or $G=\Sigma_{q_1}\oplus \cdots \oplus \Sigma_{q_r}$, Corollary \ref{cor:Main} follows immediately from Theorem \ref{thm:Main0} and Lemma \ref{lem:Translate}.  If $G$ does not have this form, it can be written as $F\times H$, where $F$ is finite and $H=\Sigma_{q_1}\oplus \cdots \oplus \Sigma_{q_r}$.  If $S\subseteq G$ is measure expanding, then Lemma \ref{lem:FiniteIndex} implies that $(S-g)\cap H$ is a set of recurrence for every $g\in G$.  Let $g_{1},\dots, g_{l}$ be a collection of coset representatives of $H$, and for each $i$, apply Theorem \ref{thm:MainFinite} and Lemma \ref{lem:Translate} to choose a set $S_{i}'\subseteq S\cap (H+g_{i})$ so that $(S_{i}'-g)\cap H$ is a set of recurrence for every $g\in H+g_{i}$, while no translate of $(S_{i}'-g)\cap H$ is a set of strong recurrence.   Thus for each $i$, no translate of $S_{i}'$ is a set of strong recurrence.  Setting $S'=\bigcup_{i=1}^{l} S_{i}'$, we have that every translate of $S'$ is a set of recurrence, while Lemma \ref{lem:Divisibility} implies no translate of $S'$ is a set of strong recurrence.
\end{proof}

\bibliographystyle{amsplain}
\frenchspacing
\bibliography{generic_density}

\providecommand{\bysame}{\leavevmode\hbox to3em{\hrulefill}\thinspace}
\providecommand{\MR}{\relax\ifhmode\unskip\space\fi MR }
\providecommand{\MRhref}[2]{%
  \href{http://www.ams.org/mathscinet-getitem?mr=#1}{#2}
}
\providecommand{\href}[2]{#2}
\begin{thebibliography}{10}

\bibitem{Ackelsberg}
Ethan~M. Ackelsberg, \emph{Rigidity, weak mixing, and recurrence in abelian
  groups}, Discrete Contin. Dyn. Syst. \textbf{42} (2022), no.~4, 1669--1705.
  \MR{4385772}

\bibitem{BadeaGrivauxMatheron}
Catalin Badea, Sophie Grivaux, and \'{E}tienne Matheron, \emph{Rigidity
  sequences, {K}azhdan sets and group topologies on the integers}, J. Anal.
  Math. \textbf{143} (2021), no.~1, 313--347. \MR{4299163}

\bibitem{BBF}
Mathias Beiglb\"{o}ck, Vitaly Bergelson, and Alexander Fish, \emph{Sumset
  phenomenon in countable amenable groups}, Adv. Math. \textbf{223} (2010),
  no.~2, 416--432. \MR{2565535}

\bibitem{BjorklundFishPlunnecke}
Michael Bj\"{o}rklund and Alexander Fish, \emph{Pl\"{u}nnecke inequalities for
  countable abelian groups}, J. Reine Angew. Math. \textbf{730} (2017),
  199--224. \MR{3692018}

\bibitem{BKQW}
Michael Boshernitzan, Grigori Kolesnik, Anthony Quas, and M\'at\'e Wierdl,
  \emph{Ergodic averaging sequences}, J. Anal. Math. \textbf{95} (2005),
  63--103. \MR{2145587}

\bibitem{CherryKQ}
Cash Cherry, \emph{Rigid-recurrent sequences for actions of finite exponent
  groups}, submitted.

\bibitem{FanSchneiderRecurrence}
AiHua Fan and Dominique Schneider, \emph{Recurrence properties of sequences of
  integers}, Sci. China Math. \textbf{53} (2010), no.~3, 641--656. \MR{2608321}

\bibitem{FaKa}
Bassam Fayad and Adam Kanigowski, \emph{Rigidity times for a weakly mixing
  dynamical system which are not rigidity times for any irrational rotation},
  Ergodic Theory Dynam. Systems \textbf{35} (2015), no.~8, 2529--2534.
  \MR{3456605}

\bibitem{Folland}
Gerald~B. Folland, \emph{A course in abstract harmonic analysis}, Studies in
  Advanced Mathematics, CRC Press, Boca Raton, FL, 1995.

\bibitem{FolnerBanach}
Erling F{\o}lner, \emph{On groups with full {B}anach mean value}, Math. Scand.
  \textbf{3} (1955), 243--254. \MR{79220}

\bibitem{Forrest}
A.~H. Forrest, \emph{The construction of a set of recurrence which is not a set
  of strong recurrence}, Israel J. Math. \textbf{76} (1991), no.~1-2, 215--228.

\bibitem{ForrestThesis}
Alan~H. Forrest, \emph{Recurrence in dynamical systems: {A} combinatorial
  approach}, ProQuest LLC, Ann Arbor, MI, 1990, Thesis (Ph.D.)--The Ohio State
  University. \MR{2685439}

\bibitem{Fuchs}
L\'aszl\'o Fuchs, \emph{Abelian groups}, Springer Monographs in Mathematics,
  Springer, Cham, 2015. \MR{3467030}

\bibitem{Fdiagonal}
Harry Furstenberg, \emph{Ergodic behavior of diagonal measures and a theorem of
  {S}zemer\'edi on arithmetic progressions}, J. Analyse Math. \textbf{31}
  (1977), 204--256. \MR{0498471}

\bibitem{FurstenbergPoincare}
\bysame, \emph{Poincar\'{e} recurrence and number theory}, Bull. Amer. Math.
  Soc. (N.S.) \textbf{5} (1981), no.~3, 211--234. \MR{628658}

\bibitem{GriesmerIsr}
John~T. Griesmer, \emph{Sumsets of dense sets and sparse sets}, Israel J. Math.
  \textbf{190} (2012), 229--252. \MR{2956240}

\bibitem{GrPop}
\bysame, \emph{Recurrence, rigidity, and popular differences}, Ergodic Theory
  and Dynamical Systems (2017), 1--18.

\bibitem{KaplanskyBook}
Irving Kaplansky, \emph{Infinite abelian groups}, revised ed., University of
  Michigan Press, Ann Arbor, MI, 1969. \MR{233887}

\bibitem{McCutcheonAlexandria}
Randall McCutcheon, \emph{Three results in recurrence}, Ergodic theory and its
  connections with harmonic analysis ({A}lexandria, 1993), London Math. Soc.
  Lecture Note Ser., vol. 205, Cambridge Univ. Press, Cambridge, 1995,
  pp.~349--358. \MR{1325710}

\bibitem{MorrisLCA}
Sidney~A. Morris, \emph{Pontryagin duality and the structure of locally compact
  abelian groups}, Cambridge University Press, Cambridge-New York-Melbourne,
  1977, London Mathematical Society Lecture Note Series, No. 29. \MR{0442141}

\bibitem{NairPoincare1}
R.~Nair, \emph{On uniformly distributed sequences of integers and
  {P}oincar\'{e} recurrence}, Indag. Math. (N.S.) \textbf{9} (1998), no.~1,
  55--63. \MR{1618231}

\bibitem{Niederreiter}
H.~Niederreiter, \emph{On a paper of {B}lum, {E}isenberg, and {H}ahn concerning
  ergodic theory and the distribution of sequences in the {B}ohr group}, Acta
  Sci. Math. (Szeged) \textbf{37} (1975), 103--108. \MR{396455}

\bibitem{Rudin}
Walter Rudin, \emph{Fourier analysis on groups}, Interscience Tracts in Pure
  and Applied Mathematics, No. 12, Interscience Publishers (a division of John
  Wiley and Sons), New York-London, 1962. \MR{0152834 (27 \#2808)}

\bibitem{Saeki}
Sadahiro Saeki, \emph{Bohr compactification and continuous measures}, Proc.
  Amer. Math. Soc. \textbf{80} (1980), no.~2, 244--246. \MR{577752 (81i:43008)}

\bibitem{Sarkozy}
A.~S\'ark\H{o}zy, \emph{On difference sets of sequences of integers. {I}}, Acta
  Math. Acad. Sci. Hungar. \textbf{31} (1978), no.~1--2, 125--149. \MR{0466059}

\bibitem{Wierdl}
Mate Wierdl, \emph{Almost everywhere convergence and recurrence along
  subsequences in ergodic theory}, ProQuest LLC, Ann Arbor, MI, 1989, Thesis
  (Ph.D.)--The Ohio State University. \MR{2638457}

\bibitem{Wolf}
J.~Wolf, \emph{The structure of popular difference sets}, Israel J. Math.
  \textbf{179} (2010), 253--278.

\end{thebibliography}
\end{document}